\documentclass[twoside,11pt]{article}

\usepackage{jmlr2e}

\usepackage{graphicx}
\usepackage{booktabs} 
\usepackage{subcaption} 

\usepackage{hyperref} 
\usepackage{fancybox}
\usepackage[square,numbers]{natbib}

\def \S {\mathbf{S}}

\def \Y {\mathcal{Y}}

\def \R {\mathbb{R}}

\def \w {\mathbf{w}}
\def \v {\mathbf{v}}

\def \x {\mathbf{x}}

\def \E {\mathbb{E}}

\def \x {\mathbf{x}}
\def \p {\mathbf{p}}
\def \a {\mathbf{a}}

\def \e {\mathbf{e}}

\def \d {\mathbf{d}}

\def \1 {\mathbf{1}} 

\def \z {\mathbf{z}}

\def \y {\mathbf{y}}
\def \u {\mathbf{u}}

\def \I {\mathbf{I}}
\def \P {\mathcal{P}}
\def \Q {\mathcal{Q}}

\def \D {\mathcal{D}}

\usepackage{soul} 
\usepackage{graphicx} 
\usepackage{amsmath}
\usepackage{booktabs}
\usepackage{algorithm}
\usepackage{algorithmic}

\usepackage{amsfonts} 
\usepackage{verbatim} 
\usepackage{amssymb}
\usepackage{verbatim}
\usepackage{enumitem}
\usepackage{sidecap,color}
\newcommand{\required}[1]{\section*{\hfil #1\hfil}}

\newtheorem{thm}{Theorem}
\newtheorem{ass}{Assumption}

\begin{document}

\title{Randomized Stochastic Variance-Reduced  Methods for Multi-Task  Stochastic Bilevel Optimization}

\author{\name{Zhishuai Guo$^\dagger$}\email{zhishuai-guo@uiowa.edu}\\ 
\name{Quanqi Hu$^\dagger$}\email{quanqi-hu@uiowa.edu}\\ 
\name{Lijun Zhang$^\ddagger$}\email{zhanglj@lamda.nju.edu.cn}\\ 
 \name{Tianbao Yang$^\dagger$} \email{tianbao-yang@uiowa.edu}\\
\addr$^\dagger$Department of Management  Sciences,
   The University of Iowa, Iowa City, IA 52242 \\  
\addr$^\ddagger$ Nanjing University, Nanjing 210023, China\\
}

\maketitle

\begin{abstract} 
In this paper, we consider non-convex stochastic bilevel optimization (SBO) problems that have many applications in machine learning. Although numerous  studies have proposed stochastic algorithms for solving these problems, they are limited in two perspectives: (i) their sample complexities are high, which do not match the state-of-the-art result for non-convex stochastic optimization; (ii) their algorithms are tailored to problems with only one lower-level problem. When there are many lower-level problems, it could be prohibitive to process all these lower-level problems at each iteration. To address these limitations, this paper proposes fast randomized stochastic algorithms for non-convex SBO problems. First, we present a stochastic method for non-convex SBO with only one lower problem and establish its sample complexity of $O(1/\epsilon^3)$ for finding an $\epsilon$-stationary point under Lipschitz continuous conditions of stochastic oracles, matching the lower bound for stochastic smooth non-convex optimization. Second, we present a randomized stochastic method for non-convex SBO with $m>1$ lower level problems (multi-task SBO) by processing a constant number of lower problems at each iteration,  and establish its sample complexity no worse than  $O(m/\epsilon^3)$, which could be a better complexity than that of simply processing all $m$ lower problems at each iteration. Lastly, we  establish even faster convergence results for gradient-dominant functions.   To the best of our knowledge, this is the first work considering multi-task SBO  and developing state-of-the-art sample complexity results.
\end{abstract}

\section{Introduction}
A stochastic bilevel optimization (SBO) problem is formulated as: 
\begin{equation}\label{eqn:sbo}
\begin{aligned}
&\min_{\x\in\R^d} F(\x) = f(\x, \y^*(\x)) = \E_{\xi}[f(\x, \y^*(\x); \xi)],\quad\text{ (upper)}\\
&\y^*(\x) \in \arg\min_{\y\in \R^{d'}} g(\x, \y)  =\E_{\zeta}[g(\x, \y; \zeta)], \hspace*{0.12in}\quad\quad\text{ (lower)}
\end{aligned}
\end{equation}
where $f$ and $g$ are continuously differentiable functions, $\xi$ and $\zeta$ are random variables. This problem has many  applications in machine learning, e.g.,  meta-learning~\citep{DBLP:conf/nips/RajeswaranFKL19}, neural architecture search~\citep{liu2018darts}, reinforcement learning~\citep{journals/siamco/KondaB99}, hyperparameter optimization~\citep{pmlr-v80-franceschi18a}, etc. 

Besides the above problem,  we  consider a more general SBO problem with many lower problems in the following form: 
\begin{equation}\label{eqn:msbo}
\begin{aligned}
&\min_{\x\in \R^d} F(\x) = \frac{1}{m}\sum_{i=1}^mf_i(\x, \y_i^*(\x)) = \frac{1}{m}\sum_{i=1}^m\E_{\xi\sim \mathcal P_i}[f_i(\x, \y_i^*(\x); \xi)],\quad\text{ (upper)}\\
&\y_i^*(\x) \in \arg\min_{\y_i\in \R^{d_i}} g_i(\x, \y_i)  =\E_{\zeta\sim \Q_i}[g_i(\x, \y_i; \zeta)], i=1, \ldots, m, \hspace*{0.04in}\quad\quad\text{ (lower)}
\end{aligned}
\end{equation}
where $m\gg 1$. We refer to the above problem as \textbf{multi-task stochastic bilevel optimization}. One challenge of dealing with the problem~(\ref{eqn:msbo}) compared with the problem~(\ref{eqn:sbo}) is that processing $m$ lower problems at each iteration could be prohibitive when $m$ is large.

\begin{table*}[t] 
	\caption{Sample Complexity of Different Algorithms for solving~(\ref{eqn:msbo}), where $C(\p) = \min_{\p\in\Delta}\sum_{i=1}^m C_i /p_i$ with $C_i$ being a constant related to $f_i(\x, y)$ and $g_i(\x, \y)$, $\Delta=\{\p\in\R^m; \sum_ip_i=1, p_i\geq 0\}$. \# of Tasks per-iteartion denotes that how many lower problems are accessed per-iteration. 
}\label{tab:1} 
	\centering
	\label{tab:2}
	\scalebox{1}{\begin{tabular}{llccc}
			\toprule
			&batch size&\# of loops& Sample Complexity&	\# of Tasks per-iteration\\
		RSVRB& $O(1)$&Single&$O(C(\p)/(m\epsilon^3))$&$O(1)$\\
		STABLE~\cite{DBLP:journals/corr/abs-2102-04671}&$O(1)$&Single&$O(m/\epsilon^4)$&$O(m)$\\
		TTSA ~\cite{DBLP:journals/corr/abs-2007-05170}&  $O(1)$&Single&$O(m/\epsilon^5)$&$O(m)$\\
		StocBio~\cite{DBLP:journals/corr/abs-2010-07962}&$O(1/\epsilon^2)$& Double&$O(m/\epsilon^4)$&$O(m)$ \\
		BSA~\cite{99401}&$O(1)$& Double&$O(m/\epsilon^6)$ &$O(m)$\\
		\bottomrule
	\end{tabular}}
	\vspace*{-0.15in} 
\end{table*}

\paragraph{Related Work.} Although bilevel optimization has a long history in the literature~\citep{Colson07anoverview,doi:10.1080/10556780802102586,Kunisch13abilevel,pmlr-v119-liu20l,doi:10.1137/16M105592X,Shaban2019TruncatedBF}, non-asymptotic convergence was not established until recently~\cite{99401,DBLP:journals/corr/abs-2010-07962,DBLP:journals/corr/abs-2007-05170,DBLP:journals/corr/abs-2102-04671}. \cite{99401} proposes a double-loop stochastic algorithm, where the inner loop solves the lower level problem up to a certain accuracy level. For non-convex $f$ and strongly convex $g$, their algorithm's sample complexity is $O(1/\epsilon^6)$ for finding an $\epsilon$-stationary point of the objective function $F$, i.e., a solution $\x$ satisfying $\E[\|\nabla F(\x)\|]\leq \epsilon$. \cite{DBLP:journals/corr/abs-2010-07962} improves the sample complexity of a double-loop algorithm to $\widetilde O(1/\epsilon^4)$ by using a large mini-batch size at each iteration. \cite{DBLP:journals/corr/abs-2007-05170} proposes a single-loop two timescale algorithm TTSA, which suffers from a worse complexity of $O(1/\epsilon^5)$ for  finding an $\epsilon$-stationary point.  Recently, \cite{DBLP:journals/corr/abs-2102-04671} proposes  a single-loop single timescale algorithm (STABLE) based on a recursive variance-reduced estimator of the second-order Jacobian $\nabla_{xy}^2 g(\x, \y)$ and Hessian matrix $\nabla_{yy}^2\nabla g(\x, \y)$, which enjoys a sample complexity of $O(1/\epsilon^4)$ without using a large mini-batch size at each iteration. 
However, these  studies  focus  on the formulation~(\ref{eqn:sbo}) and do not explicitly consider the challenge for dealing with~(\ref{eqn:msbo}) with many lower level problems.  Hence, the existing results are limited in two perspectives: (i) their sample complexities are high, which do not match the state-of-the-art result for non-convex stochastic optimization, i.e., $O(1/\epsilon^3)$~\cite{DBLP:journals/corr/abs-1807-01695,cutkosky2019momentum}; (ii) their algorithms are tailored to problems with only one lower-level problem and are not appropriate for solving~(\ref{eqn:msbo}) with $m\gg 1$.

\textbf{Our Contributions.} In this paper, we further improve the sample complexity of optimization algorithms for solving SBO problems. First, we present a stochastic method for non-convex SBO with only one lower problem and establish a  sample complexity of $O(1/\epsilon^3)$ for finding an $\epsilon$-stationary point under Lipschitz continuous conditions of stochastic oracles, matching a lower bound for stochastic smooth non-convex optimization~\cite{arjevani2019lower}. Second, we present a randomized stochastic method (RSVRB) for non-convex multi-task SBO with $m>1$ lower level problems by processing only a constant number of lower problems at each iteration,  and establish its sample complexity no worse than  $O(m/\epsilon^3)$, which could have a better complexity than simply processing all $m$ lower problems at each iteration. Lastly, we establish even faster convergence rates in the order of $O(1/\epsilon)$ for gradient-dominant functions.   To the best of our knowledge, this is the first work considering SBO with many lower level problems and establishing state-of-the-art sample complexity. We present the comparison between our results and existing results in Table~\ref{tab:1}.


\paragraph{Applications \& Significance.} Although this paper focuses on theoretical developments,  the proposed algorithms have broad applications in machine learning. We name some of these applications below and discuss the significance of our results. 

{\bf Meta Learning.} An important application of SBO with many lower level problems is model-agnostic meta learning (MAML) with many tasks~\citep{DBLP:conf/nips/RajeswaranFKL19}, which can be formulated as following: 
\begin{align*}
&\min_{\x\in \R^d}\sum_{i=1}^mf_i(\x+ \y_i^*(\x))\\
&\y_i^*(\x)\in\arg\min_{\y_i} g_i(\x + \y_i) + \frac{\lambda}{2}\|\y_i\|^2, i = 1, \ldots, m,
\end{align*}
where $f_i$ corresponds to the (expectation or average) loss function for the $i$-th task, and  $\y_i^*(\x)$ denotes the personalized model increment for the $i$-th task by minimizing the lower problem. Existing algorithms~\cite{99401,DBLP:journals/corr/abs-2010-07962,DBLP:journals/corr/abs-2007-05170,DBLP:journals/corr/abs-2102-04671,DBLP:conf/nips/RajeswaranFKL19} for bilevel optimization do not consider the scenario that $m$ is very large. Hence, these algorithms require processing $m$ tasks at each iteration. In contrast, the proposed RSVRB only needs to process one task per-iteration and only requires a constant mini-batch size and enjoys a smaller sample complexity under appropriate conditions. This also improves existing stochastic algorithms and their theory for MAML~\cite{fallah2020convergence,DBLP:conf/nips/HuZCH20}, which require processing a large number of tasks or a large mini-batch size per-iteration and suffer from worse complexities with the best complexity of $O(1/\epsilon^5)$ under similar conditions as ours~\cite{DBLP:conf/nips/HuZCH20}. 

{\bf Optimizing a Sum of Stochastic Compositional Functions.} 
Existing literature has considered the following compositional optimization problem: 
\begin{align}\label{eqn:comp}
\min_{\x\in\R^d}f(g(\x)) =  \E_\xi[f(\E_{\zeta}[g(\x; \zeta)]; \xi)], 
\end{align}
which is a special case of SBO, i.e., 
\begin{align*}
&\min_{\x\in\R^d}f(\y^*(\x)) = \E[f(\y^*(\x); \xi)]\\
& \y^*(\x)=\arg\min  - \y^{\top}\E_{\zeta}[g(\x; \zeta)] + \frac{1}{2}\|\y\|^2.
\end{align*}
Although the two-level stochastic compositional optimization problem~(\ref{eqn:comp}) has been studied extensively in the literature~\cite{wang2017stochastic,ghadimi2020single,zhang2019stochastic, DBLP:conf/aaai/HuoGLH18,zhou2019momentum,DBLP:journals/corr/abs-1809-02505}, the following extension with many components has been less studied: 
\begin{align}\label{eqn:mcomp}
\min_{\x\in\R^d}\sum_{i=1}^mf_i(g_i(\x)) =  \sum_{i=1}^m\E_{\xi\sim\P_i}[f_i(\E_{\zeta\sim \Q_i}[g_i(\x; \zeta)]; \xi)],
\end{align}
which is a special case of SBO with many lower-level problems: 
\begin{align*}
&\min_{\x\in\R^d}\sum_{i=1}^mf_i(\y_i^*(\x)) =\sum_{i=1}^m \E_{\xi\sim \P_i}[f_i(\y_i^*(\x); \xi)]\\
& \y_i^*(\x)=\arg\min  - \y_i^{\top}\E_{\zeta\sim \Q_i}[g_i(\x; \zeta)];+ \frac{1}{2}\|\y_i\|^2, i=1, \ldots, m. 
\end{align*}
The problem~(\ref{eqn:mcomp}) has important applications in machine learning. We give one example, which is the optimization of area under precision-recall curve (AUPRC). Recently, \cite{DBLP:journals/corr/abs-2104-08736} proposes to optimize a surrogate loss of averaged precision for AUPRC optimization, which can be formulated as 
\begin{align}\label{eqn:auprc}
\min_{\x\in\R^d}\sum_{b_i=1}\frac{-\E_{\a_j, b_j}[\ell(\x; \a_i, \a_j)\I(b_j=1)]}{\E_{\a_j, b_j}[\ell(\x; \a_i, \a_j)]} ,
\end{align}
where $\x$ denotes a model parameter, $(\a_i, b_i)$ denote an input data and its corresponding label $b_i\in\{1, -1\}$, and $\ell(\w; \a_i, \a_j)$ denotes a pairwise loss function. Define $f_i (\mathbf g) = \frac{-g_{i,1}(\x)}{g_{i,2}(\x)}$, $g_{i,1}(\x) = \E_{\a_j, b_j}[\ell(\x; \a_i, \a_j)\I(b_j=1)]$ and $g_{i,2}(\x) = \E_{\a_j, b_j}[\ell(\x; \a_i, \a_j)]$, the above problem reduces to an instance of~(\ref{eqn:mcomp}) and hence an instance of~(\ref{eqn:msbo}). \cite{DBLP:journals/corr/abs-2104-08736} has proposed a stochastic algorithm for solving~(\ref{eqn:auprc}), whose sample complexity is $O(m/\epsilon^5)$. Under a Lipschitz continuity condition of $\ell(\x; \a, b)$ and other appropriate conditions, the proposed  RSVRB method has a better sample complexity of $O(m/\epsilon^3)$.

{\bf Structured Non-Convex Strongly-Concave Min-Max Optimization.} 
The proposed RSVRB can be applicable to the following structured min-max problem: 
\begin{align}\label{eqn:minmax}
\min_{\x\in\R^d}\max_{\y_i\in\Omega_i \atop i=1, \ldots, m}\sum_{i=1}^mf_i(\x, \y_i) \Longleftrightarrow \min_{\x\in\R^d}\sum_{i=1}^mf_i(\x, \y_i(\x^*)),\: s.t.\: \y_i^*(\x) = \arg\min_{\y_i\in\Omega_i} - f_i(\x, \y_i), i\in[m].
\end{align}
For the problem~(\ref{eqn:minmax}) with $m=1$, our algorithm SVRB reduces to a stochastic variance-reduced algorithm for solving non-convex strongly-concave min-max problem proposed in~\cite{DBLP:journals/corr/abs-2008-08170}. However, when $m$ is large, our algorithm RSVRB could be more efficient. One application of~(\ref{eqn:minmax}) with large $m$ is multi-task deep AUC maximization where each $i$ corresponds to a binary classification task since AUC maximization for one task can be written as a non-convex strongly-concave min-max optimization problem with one dual variable~\cite{liu2019stochastic,DBLP:journals/corr/abs-2012-03173}. 

\section{Preliminaries}
\paragraph{Notations.} We use $\|\cdot\|$ to denote the $\ell_2$ norm of a vector or a matrix. The $\ell_2$ norm of a matrix is known as the spectral norm. Let $\|\cdot\|_F$ denote the Frobenius norm of a matrix. In the following, an Euclidean norm refers to $\ell_2$ norm of a vector or Frobenius norm of a matrix. Let $\Pi_{\Omega}[\cdot]$ denotes Euclidean projection onto a convex set $\Omega$ for a vector and $\S^l_{\lambda}[X]$ denotes a projection onto the set $\{X\in \R^{d\times d}: X\succeq \lambda I\}$.  We use the notation $\S_C^u[X]$ denote a projection onto the set $\{X\in \R^{d\times d}: \|X\|\leq C\}$. Both $\S^l_{\lambda}[X]$ and $\S^u_{C}[X]$ can be implemented by using SVD and thresholding the singular values. For an Euclidean ball $\Omega_C=\{\u\in\R^d, \|\u\|\leq C\}$, we abuse the notation $\Pi_{C} = \Pi_{\Omega_C}$.

	\setlength{\textfloatsep}{10pt}

Let us first present some preliminaries for solving SBO with only one lower-level problem. Recall the problem:
\begin{equation}
\begin{aligned}
&\min_{\x} F(\x) = f(\x, \y^*(\x)) =\E_\xi[ f(\x, \y^*(\x); \xi)],\\
&\y^*(\x) \in \arg\min g(\x, \y) = \E_\zeta[g(\x, \y; \zeta)].
\end{aligned}
\end{equation}

We make the following assumptions about the above problem. 
\begin{ass}\label{ass:1}
For any $\x$, $g(\x, \y)$ is $L_g$-smooth and $\lambda$-strongly convex function, i.e., $L_g I\succeq\nabla_{yy}^2 g(\x, \y)\succeq \lambda I$. 
\end{ass}
\begin{ass}\label{ass:2}
For $f, g$ we assume the following conditions hold 
\begin{itemize}
\item (i) $\nabla_x f(\x, \y; \xi)$ is $L_{fx}$-Lipschitz continuous, $\nabla_y f(\x, \y; \xi)$ is $L_{fy}$-Lipschitz continuous, $\nabla_y g(\x, \y; \zeta)$ is $L_{gy}$-Lipschitz continuous, $\nabla_{xy}^2 g(\x, \y; \zeta)$ is $L_{gxy}$-Lipschitz continuous, $\nabla_{yy}^2 g(\x, \y; \zeta)$ is $L_{gyy}$-Lipschitz continuous, all respect to $(\x, \y)$. 
\item (ii)  $\nabla_x f(\x, \y; \xi)$, $\nabla_y f(\x, \y; \xi)$, $\nabla_y g(\x, \y; \zeta)$, $\nabla_{xy}^2 g(\x, \y; \zeta)$ and $\nabla_{yy}^2 g(\x, \y; \zeta)$  have a variance bounded by $\sigma^2$. 
\item (iii)  $\|\nabla_y f(\x, \y)\|^2]\leq C_{fy}^2$, $\|\nabla_{xy}^2 g(\x, \y)\|^2]\leq C_{gxy}^2$.
\item (iv) $F(\x_0) - F(\x_*) \leq \Delta$, where $\x_* = \arg\min\limits_{\x} F(\x)$. 
\end{itemize}
\end{ass}
{\bf Remark:} Assumption~\ref{ass:1} is made in many existing works for SBO~\cite{99401,DBLP:journals/corr/abs-2010-07962,DBLP:journals/corr/abs-2007-05170,DBLP:journals/corr/abs-2102-04671}. Assumptions~\ref{ass:2} (ii)$\sim$(iv) are also standard in the literature~\cite{99401,DBLP:journals/corr/abs-2010-07962,DBLP:journals/corr/abs-2007-05170}.  In terms of Assumption~\ref{ass:2} (i) compared with~\cite{DBLP:journals/corr/abs-2102-04671}, we assume stronger conditions that $\nabla_x f(\x, \cdot, \xi)$ for any fixed $\x$ and $\nabla_y g(\cdot, \cdot, \zeta)$ are Lipschitz continuous; in contrast, they assume $\nabla_x f(\x, \cdot)$ for any fixed $\x$ and $\nabla_y g(\cdot, \cdot)$ are Lipschitz continuous. However, we have weaker assumptions $\|\nabla_y f(\x, \y)\|^2]\leq C_{fy}^2$, $\|\nabla_{xy}^2 g(\x, \y)\|^2]\leq C_{gxy}^2$. In contrast, they assume $\E\|\nabla_y f(\x, \y; \xi)\|^p]\leq C_{fy}^p$, and $\E\|\nabla_x f(\x, \y; \xi)\|^p]\leq C_{fx}^p$ for both $p=2, 4$,  and $\E\|\nabla_{xy}^2 g(\x, \y)\|^2\leq C_{gxy}^2$ and $\E\|\nabla_{yy}^2 g(\x, \y)\|^2\leq C_{gyy}^2$.  It is also notable that \citep{99401,DBLP:journals/corr/abs-2007-05170} make an additional (implicit) assumption that $L_{gy} I\succeq \nabla_{yy}^2 g(\x, \y; \zeta)\succeq \lambda I$ (cf. the proof of Lemma 3.2 in~\citep{99401}). Notice that this is a much stonger assumption, which usually does not hold in practice. 

In order to understand the proposed algorithms, we present following proposition about the gradient of $F(\x)$, which is a standard result in the literature of bilevel optimization~\citep{99401}.  
\begin{proposition}\label{prop:1}
Under Assumption~\ref{ass:1}, we have
\begin{align*}
\nabla F(\x) &  = \nabla_x f(\x, y^*(\x)) + \nabla y^*(\x)^{\top}\nabla_y f(\x, \y^*(\x))\\
& =  \nabla_x f(\x, \y^*(\x))  - \nabla_{xy}^2 g(\x, \y^*(\x))[\nabla_{yy}^2g(\x, \y^*(\x))]^{-1}\nabla_y f(\x, \y^*(\x)).
\end{align*}
\end{proposition}


\section{SBO with One Lower-level Problem}

{\bf Motivation of the Proposed Algorithm.} A naive approach to approximate the gradient $\nabla F(\x)$ at any iterate $\x_t$ is to compute $\y^*(\x_t)$ and then approximate each term in the R.H.S of the equation in Proposition~\ref{prop:1} by their (mini-batch) stochastic gradients, i.e., $\nabla_x f(\x_t, \y^*(\x_t); \xi_1)  - \nabla_{xy}^2 g(\x_t, \y^*(\x_t); \xi_{21})[\nabla_{yy}^2g(\x_t, \y^*(\x_t); \xi_{22})]^{-1}\nabla_y f(\x_t, \y^*(\x_t); \xi_1)$, where $\xi_1, \xi_{21}, \xi_{22}$ are independent stochastic random variables.  There are two deficiencies of this naive approach: (i) Due to the inverse function, the expectation of above stochastic variable is not equal to $\nabla F(\x_t)$. In order to control the error of the stochastic gradient estimator, a straightforward approach is to use a mini-batch stochastic gradient estimator for $\nabla_{yy}^2g(\x_t, \y^*(\x_t))$ with a large mini-batch size. A better way is to use a variance-reduced estimator for $\nabla_{yy}^2g(\x_t, \y^*(\x_t))$, which has been studied extensively in the literature of stochastic optimization and stochastic compositional optimization~\citep{DBLP:journals/corr/abs-2006-10138,cutkosky2019momentum,chen2020solving,DBLP:journals/mp/WangFL17,DBLP:journals/siamjo/GhadimiRW20}.  The inverse function $[\nabla_{yy}^2g(\x, \y^*(\x))]^{-1}$ in $\nabla F(\x)$ makes it similar to a gradient of a compositional function. Hence, an estimator of $\nabla_{yy}^2g(\x_t, \y^*(\x_t))$ with variance-reduced property is important to the convergence of an algorithm without using a large mini-batch size. (ii) Computing $\y^*(\x_t)$ exactly at each iteration is computationally expensive and is also not necessary. This is because at the beginning $\x_t$ is far from good, hence computing an approximate solution of $\y^*(\x_t)$ is enough for the algorithm to make the progress.  

Based on the above motivation, at each iteration $t$, we use $\y_t$ to approximate $\y^*(\x_t)$, and use variance-reduced stochastic gradient estimators to estimate each term in the following: 
\begin{align*}
\nabla F(\x_t, \y_t) = \nabla_x f(\x_t, \y_t) - \nabla^2_{xy}g(\x_t, \y_t)[\nabla_{yy}^2g(\x_t, \y_t)]^{-1}\nabla_y f(\x_t, \y_t).
\end{align*}
To this end, we will maintain four quantities $\u_{t+1}, \v_{t+1}, V_{t+1}, H_{t+1}$ for estimating $ \nabla_x f(\x_t, \y_t)$, $\nabla_y f(\x_t, \y_t)$, $ \nabla^2_{xy}g(\x_t, \y_t)$, $\nabla_{yy}^2g(\x_t, \y_t)$, respectively, and estimate $\nabla F(\x_t, \y_t)$ by  
\begin{align*}
\z_{t+1} = \u_{t+1} -V_{t+1}[H_{t+1}]^{-1}\v_{t+1}.
\end{align*}
Then we can update $\x_{t+1}$ by $\x_{t+1} = \x_t - \eta_{t} \z_{t+1}$. 
\begin{algorithm}[t]
\caption {Stochastic Variance-Reduced Bilevel Method: SVRB}\label{alg:1}
\begin{algorithmic}[1]
\STATE{Initialization: $\x_0\in \R^d, \y_0 \in\R^{d'}, \u_0, \v_0, V_0, H_0, \w_0$}
\FOR{$t=1,2, ..., T$}
\STATE {$\u_{t+1} = (1-\beta_{fx,t})(\u_t - \nabla_x f(\x_{t-1}, \y_{t-1}; \xi_t))  + \nabla_x f(\x_t, \y_t; \xi_t) $}
\STATE {$\v_{t+1} = \Pi_{C_{fy}}[(1-\beta_{fy,t})(\v_t - \nabla_y f(\x_{t-1}, \y_{t-1}; \xi_t))+  \nabla_y f(\x_t, \y_t; \xi_t)]$}
\STATE $V_{t+1} = \S^u_{C_{gxy}}[(1-\beta_{gxy,t})(V_t - \nabla_{xy}^2g(\x_{t-1}, \y_{t-1}; \zeta_t) +\nabla_{xy}^2g(\x_{t}, \y_{t}; \zeta_t) ]$
\STATE $H_{t+1} = \S^l_{\lambda}[(1-\beta_{gyy,t})(H_t - \nabla_{yy}^2g(\x_{t-1}, \y_{t-1}; \zeta_t) +\nabla_{yy}^2g(\x_{t}, \y_{t}; \zeta_t) ]$
\STATE {$\w_{t+1} = (1-\beta_{gy,t})(\w_t - \nabla_y g(\x_{t-1}, \y_{t-1}; \zeta_t))+  \nabla_y g(\x_t, \y_t; \zeta_t)$}
\STATE $\z_{t+1} = \u_{t+1} - V_{t+1}[H_{t+1}]^{-1}\v_{t+1}$
\STATE$\x_{t+1} = \x_{t} - \eta_t\gamma\z_{t+1}$
\STATE {$\y_{t+1} = \y_t - \tau_t\tau \w_{t+1}$} 
\ENDFOR 
\STATE{\textbf{return} $(\x_{\Tilde{t}}, y_{\Tilde{t}})$ for a randomly selected $\Tilde{t}\in\{0, 1, \ldots, T\}$.}
\end{algorithmic}
\label{alg:main}
\end{algorithm}

Next, we consider the update for $\y_{t+1}$. This update is for approximating $\y^*(\x_{t+1})$. A simple approach is to simply use the stochastic gradient of $g(\x, \y)$ at previous solutions $(\x_t, \y_t)$. However, in order to match the variance-reduced property of $\z_{t+1}$, we will also use a variance-reduced gradient estimator for $\nabla g_y(\x, \y)$ to update $\y_{t+1}$, i.e., $\y_{t+1} = \y_t - \tau_t \w_{t+1}$, where $\w_{t+1}$ is a variance-reduced estimator of $\nabla g_y(\x_t, \y_t)$. 

Based on the motivation described above, we present  the detailed steps in Algorithm~\ref{alg:1}. The algorithm is referred to as stochastic variance-reduced bilevel (SVRB) method. The key to the algorithm lies at using the STROM technique to compute the variance reduced estimators, which is developed by~\citep{cutkosky2019momentum} with the idea originated from  several earlier variance reduced estimators (SPIDER~\citep{DBLP:journals/corr/abs-1807-01695}, SARAH~\citep{DBLP:conf/icml/NguyenLST17}) for stochastic non-convex minimization. Another important feature of SVRB is the projection of $\v_{t+1}, V_{t+1}, H_{t+1}$ onto an appropriate domain to ensure appropriate boundness. 

{\bf Differences between SVRB and STABLE.} The key differences between SVRB and STABLE lie at (i) SVRB  uses a variance-reduced estimator  $\u_{t+1}, \v_{t+1}, \w_{t+1}$ with the STORM technique to approximate $\nabla_x f(\x_t, \y_t), \nabla_y f(\x_t, \y_t), \nabla_y g(\x_t, \y_t)$; in contrast STABLE uses unbiased estimators for these gradients; (ii) SVRB does a simple update for $\y_{t+1}$ using the variance-reduced estimator $\w_{t+1}$; in contrast STABLE uses  a step that includes  $H_{t+1}^{-1}V_{t+1}^{\top}(\x_{t+1} - \x_t)$ in the update of $\y_{t+1}$. These differences make the analysis of SVRB different from that of STABLE, especially on bounding $\|\y_t - \y^*(\x_t)\|^2$. Finally, we note that the $H_{t+1}^{-1}\v_{t+1}$ in Step 4 of Algorithm~\ref{alg:1} can be efficiently computed by 
the conjugate gradient (CG) algorithm that only involves Hessian-vector product~\cite{rajeswaran2019meta}. 

The main convergence result of SVRB is presented in the following theorem. 
\begin{thm}\label{thm:1}
Suppose Assumption~\ref{ass:1} and Assumption~\ref{ass:2} hold.  
With $\tau\leq 1/(3L_g)$, $\gamma\leq \min\{\frac{\sqrt{\tau\lambda}}{8\sqrt{C}L_y}, \frac{1}{16(L_{gy}^2 + L_{fx}^2 + L^2_{fy} + L^2_{gxy} + L^2_{gyy})}\}$, $\eta_t=\tau_t=c/(c_0 + t)^{1/3}$, 
$\beta_{gy, t} = (\frac{1}{7L_Fc^3}+\frac{8C\gamma\tau_t\tau}{\lambda\eta_t}) \eta_t^2$, $\beta_{fx, t+1} =  (\frac{1}{7L_Fc^3}  + C_1\gamma)\eta_t^2$, $\beta_{fy, t+1} = (\frac{1}{7L_Fc^3}  + C_4\gamma)\eta_t^2$,  $\beta_{gxy, t+1} = (\frac{1}{7L_Fc^3}  + C_2\gamma)\eta_t^2$, and $\beta_{gyy, t+1} = (\frac{1}{7L_Fc^3}  + C_3\gamma)\eta_t^2$, where
$c_0\geq\max\{2, (4 L_F c)^3, \left(\frac{c^2\lambda}{16C\gamma\tau}\right)^{3/2}, \left(\frac{2}{7L_Fc}\right)^{3/2}, (2(C_1+C_2+C_3+C_4) \gamma c^2)^{3/2}\}$, 
$C=\max\{\frac{4C_0}{\tau\lambda},\frac{4(L^2_{gy}+L^2_{fx}+L^2_{fy}+L^2_{gxy}+L^2_{gyy})}{\gamma}\}$ and 
$C_0, C_1, C_2, C_3, C_4$ 
are constants specified in the proof, we have 
\begin{align*} 
\E\bigg[\frac{1}{T+1}\sum_{t=0}^T\|\nabla F(\x_t)\|^2\bigg] \leq \widetilde O\left(\frac{1}{T^{2/3}}\right).
\end{align*}
\end{thm}
\vspace*{-0.1in}
{\bf Remark:} The above theorem implies that SVRB enjoys a sample complexity of $\widetilde O(1/\epsilon^3)$ for finding an epsilon-stationary point. The logarithmic factor can be removed by setting $\eta_t$ as a fixed value at the level of $\epsilon$. It is also notable that the step sizes for $\x_t$ and $\y_t$ are at the same level, which is referred to as single time-scale. The proof is included in the Appendix~\ref{sec:thm1}.

\section{SBO with Many Lower-level Problems}
In this section, we present an randomized algorithm by extending the algorithm in last section  for tackling SBO with many lower-level problems, i.e., 
\begin{align*}
&\min_{\x} F(\x) = \frac{1}{m}\sum_{i=1}^m \underbrace{f_i(\x, \y^*_i(\x))}\limits_{F_i(\x)}=  \E_{\xi\sim \mathcal P_i}[f_i(\x, \y^*_i(\x); \xi)],\\
&\y_i^*(\x) \in \arg\min_{\y_i\in\R^{d_i}} g_i(\x, \y_i)= \E_{\zeta\sim \Q_i}[g_i(\x, \y_i;  \zeta)], i=1, \ldots, m.
\end{align*}
Regarding the above problem, we make the following assumptions. 
\begin{ass}\label{ass:3}
For any $\x$, $g_i(\x, \y_i)$ is  $L_g$-smooth and $\lambda$-strongly convex function, i.e., $L_gI\succeq \nabla_{yy}^2 g_i(\x,  \y_i)\succeq \lambda I$. 
\end{ass}
\begin{ass}\label{ass:4}
For $f, g$ we assume the following conditions hold 
\begin{itemize}
\item $\nabla_x f_i(\x, \y_i; \xi)$ is $L_{fx,i}$-Lipschitz continuous, $\nabla_y f_i(\x, \y_i; \xi)$ is $L_{fy,i}$-Lipschitz continuous, $\nabla_y g_i(\x, \y_i; \zeta)$ is $L_{gy,i}$-Lipschitz continuous, $\nabla_{xy}^2 g_i(\x, \y_i; \zeta)$ is $L_{gxy,i}$-Lipschitz continuous, $\nabla_{yy}^2 g_i(\x, \y_i; \zeta)$ is $L_{gyy,i}$-Lipschitz continuous, all respect to $(\x, \y_i)$. 
\item $\E[\|\nabla_x f_i(\x, \y_i; \xi)\|^2]\leq C_{fx,i}^2$, $\E[\|\nabla_y f_i(\x, \y_i; \xi)\|^2]\leq C_{fy,i}^2$,  $\E[\|\nabla_{xy}^2 g_i(\x, \y_i; \zeta)\|^2]\leq C_{gxy,i}^2$, $\E[\|\nabla_{yy}^2 g_i(\x, \y_i; \zeta)\|^2]\leq C_{gyy,i}^2$. 
\item Assume  $\y^*_i(\x)$ is bounded such that $\y_i^*(\x)\in\Y_i = \{\y_i\in\R^{d_i'}: \|\y_i\|\leq R_i\}$, and $\E[\|\nabla_{y} g_i(\x, \y_i; \zeta)\|^2]\leq C_{gy,i}^2$ for $\y_i\in \Y_i$.
\item $F(\x_0) - F(\x_*) \leq \Delta$, where $\x_* = \arg\min\limits_{\x} F(\x)$.
\end{itemize}
\end{ass}
{\bf Remark:} Note that the boundness assumption is similar to~\cite{DBLP:journals/corr/abs-2102-04671} except for $\E[\|\nabla_{y} g_i(\x, \y_i \zeta)\|^2]$. In order to ensure it bounded, we have an additional assumption that the optimal solutions to the lower problems are bounded and use this to ensure the boundness condition on $\E[\|\nabla_{y} g_i(\x, \y_i \zeta)\|^2]$ holds. 

The following proposition states the gradient of $F(\x)$. 
\begin{proposition}
Under Assumption~\ref{ass:3}, we have
\begin{align*}
\nabla F(\x) &  =\frac{1}{m}\sum_{i=1}^m \left[\nabla_x f_i(\x, \y_i^*(\x)) + \nabla y_i^*(\x)^{\top}\nabla_y f_i(\x, \y_i^*(\x))\right]\\
& =\frac{1}{m}\sum_{i=1}^m \left[\nabla_x f_i(\x, \y_i^*(\x))  - \nabla_{xy}^2 g_i(\x, \y_i^*(\x))[\nabla_{yy}^2g_i(\x, \y_i^*(\x))]^{-1}\nabla_y f_i(\x, \y_i^*(\x))\right].
\end{align*}
\end{proposition}
An extra challenge for computing the gradient of $F(\x)$ is processing $m$ lower problems (sampling data and computing stochastic gradients). When $m$ is large, the cost of processing all $m$ lower problems would be prohibitive. Below, we will present an efficient algorithm for tackling bilevel optimization with many lower level problems.

\vspace*{-0.1in}
\subsection{Algorithm and Convergence}
{\bf Motivation.} From the above proposition, one might consider randomly sampling a lower problem $i_t$ and then using the gradient estimator $\z_{i,t+1}$ of $\nabla F_i(\x_t)$ from the last section for the $i$-th lower problem. However, the variance of $\z_{i,t+1}$ compared with $\nabla F(\x_t)$ cannot be well controlled to ensure fast convergence. To address this issue, we use another layer of variance-reduced update (STORM) to approximate $\nabla F(\x_t)$.

\begin{algorithm}[t]
\caption {Randomized Stochastic Variance-Reduced Bilevel Method: RSVRB}\label{alg:2}
\begin{algorithmic}[1]
\STATE{Initialization: $\x_0\in \R^d, \y_0 \in\R^{d'}, \u_0, \v_0, V_0, H_0, \w_0, \d_0, \z_0$}
\FOR{$t=1,2, ..., T$}
\STATE Sample one lower problem $i_t$ following probability $\p=(p_1, \ldots, p_m)$
\STATE { $\u_{i, t+1} =\left\{\begin{array}{ll} \Pi_{C_{fx,i}}[(1-\beta_{fx,t})(\u_{i,t} - \frac{\nabla_x f_i(\x_{t-1}, \y_{i,t-1}; \xi_{i,t})}{p_i})  + \frac{\nabla_x f_i(\x_t, \y_{i,t}; \xi_{i,t})}{p_i}] & \text{ if } i = i_t\\
(1-\beta_{fx,t})\u_{i,t}\text{ (delay update until sampled)} & \text{o.w.}  \end{array}\right.$}
\STATE { $\v_{i,t+1} = \left\{\begin{array}{ll}\Pi_{C_{fy,i}}[(1-\beta_{fy,t})(\v_{i,t} - \frac{\nabla_y f(\x_{t-1}, \y_{i,t-1}; \xi_{i,t})}{p_i})+  \frac{\nabla_y f(\x_t, \y_{i,t}; \xi_{i,t})}{p_i}]& \text{ if }i=i_t \\(1-\beta_{fy,t})\v_{i,t} \text{ (delay update until sampled)} &\text{ o.w. } \end{array}\right.$}
\STATE $V_{i,t+1} =  \left\{\begin{array}{ll}\S^u_{C_{gxy,i}}[(1-\beta_{gxy,t})(V_{i,t} - \frac{\nabla_{xy}^2g_i(\x_{t-1}, \y_{i,t-1}; \zeta_{i,t})}{p_i}) + \frac{\nabla_{xy}^2g_i(\x_{t}, \y_{i,t}; \zeta_{i,t})}{p_i} ] & \text{ if } i=i_t\\ (1-\beta_{gxy,t})V_{i,t} \text{ (delay update until sampled)} & \text{o.w.} \end{array}\right.$
\STATE $H_{i,t+1} =   \left\{\begin{array}{ll}\S^l_{\lambda}[(1-\beta_{gyy,t})(H_{i,t} - \frac{\nabla_{yy}^2g_i(\x_{t-1}, \y_{i,t-1}; \zeta_{i,t})}{p_i})+ \frac{\nabla_{yy}^2g_i(\x_{t}, \y_{i,t}; \zeta_{i,t})}{p_i} ] & \text{  if }i_t = i\\ \S^l_{\lambda}[(1-\beta_{gyy,t})H_{i,t}] \text{ (delay update until sampled)} &\text{ o.w.}\end{array}\right.$
\STATE { $\w_{i,t+1} =  \left\{\begin{array}{ll}(1-\beta_{gy,t})(\w_{i,t} - \frac{\nabla_y g_i(\x_{t-1}, \y_{i,t-1}; \xi_{i,t})}{p_i})+ \frac{ \nabla_y g_i(\x_t, \y_{i,t}; \xi_{i,t})}{p_i} & \text{  if }i = i_t\\ (1-\beta_{gy,t})\w_{i,t} \text{ (delay update until sampled)}  & \text{ o.w. }\end{array}\right.$}
\STATE $\z_{i,t+1}=  \left\{\begin{array}{ll} \u_{i,t+1} - V_{i,t+1}[H_{i,t+1}]^{-1}\v_{i,t+1} &\text{  if } i=i_t\\   \u_{i,t+1} - V_{i,t+1}[H_{i,t+1}]^{-1}\v_{i,t+1} \text{ (delay update until sampled)} & \text{ o.w.}  \end{array}\right.$
\STATE Sample $j_t\in\{1,\ldots, m\}$ randomly
\STATE If $i_{t-1}\neq j_t$, Evaluate $\u_{j_t, t}, \v_{j_t, t}, V_{j_t, t}, H_{j_t, t}$ and $\z_{j_t, t}$
\STATE If $i_{t}\neq j_t$, Evaluate $\u_{j_t, t+1}, \v_{j_t, t+1}, V_{j_t, t+1}, H_{j_t, t+1}$ and $\z_{j_t, t+1}$
\STATE $\d_{t+1} = (1-\beta_t)(\d_t - \z_{j_t, t}) + \z_{j_t, t+1} $
\STATE$\x_{t+1} = \x_{t} - \eta_t\gamma\d_{t+1}$
\STATE {$\y_{i, t+1} = \left\{\begin{array}{ll}(1-\tau_t)\y_{i,t} + \tau_t \Pi_{\Y_i}[ \y_{i,t} - \tau \w_{i,t+1}]& \text{  if } i=i_t\\ (1-\tau_t)\y_{i,t} + \tau_t \Pi_{\Y_i}[ \y_{i,t} - \tau \w_{i,t+1}] \text{ (delay update until sampled)}  & \text{ o.w. } \end{array}\right.$} 
\ENDFOR 
\STATE{\textbf{return} $(\x_{\Tilde{t}}, \y_{\Tilde{t}}, \u_{\Tilde{t}}, \v_{\Tilde{t}}, V_{\Tilde{t}}, H_{\Tilde{t}}, \w_{\Tilde{t}}, \d_{\Tilde{t}}, \z_{\Tilde{t}})$ for a randomly selected ${\Tilde{t}}\in\{0, 1, \ldots, T\}$.}
\end{algorithmic}
\end{algorithm}

Based on the above motivation, we maintain  a stochastic gradient estimator $\d_{t+1}$ and update it by 
\begin{align*}
\d_{t+1} = (1-\beta)(\d_t - \z_{j_t,t})  + \z_{j_t, t+1},
\end{align*}
where $j_t$ is a sampled lower problem at iteration $t$ for updating $\d_{t+1}$. At each iteration, we only need to make one non-trivial update of $\z_{i_t}$ for a sampled $i_t$. The detailed steps are presented in Algorithm~\ref{alg:2}. 

{\bf Explanation of  Algorithm~\ref{alg:2}.}  We let $\y_t = (\y_{1,t}, \ldots, \y_{m,t})$, $\u_t = (\u_{1,t}, \ldots, \u_{m,t})$, $\v_t = (\v_{1,t}, \ldots, \v_{m,t})$, $V_t = (V_{1,t}, \ldots, V_{m,t})$, $H_t = (H_{1,t}, \ldots, H_{m,t}), \z_t = (\z_{1,t}, \ldots, \z_{m,t})$.  Steps 4 - 9 of Algorithm~\ref{alg:2} conduct updates for $\u_t, \v_t, V_t, H_t, \w_t$. Taking $\u_t$ as an example. Its update is 
\begin{align*}
\u_{t+1} = \Pi_{C_{fx,i}}[(1-\beta_{fx, t})(\u_t - \widetilde\nabla_x f(\x_{t-1}, \y_{t-1}; \xi_{i,t}, i_t)) +  \widetilde\nabla_x f(\x_{t}, \y_{t}; \xi_{i,t}, i_t)],
\end{align*}
where $\widetilde\nabla_x f(\x_{t}, \y_{t}; \xi_{i_t,t}, i_t) = (0, \ldots, \nabla_x f_i(\x_t, \y_{i,t}; \xi_{i_t,t})/p_{i_t}, \ldots, 0)$ and $\widetilde\nabla_x f(\x_{t-1}, \y_{i,t-1}; \xi_{i_t,t}, i_t) = (0, \ldots, \nabla_x f_i(\x_{t-1}, \y_{i,t-1}; \xi_{i_t,t})/p_{i_t}, \ldots, 0)$ satisfying
\begin{align*}
&\E_{\xi, i_t}[\widetilde\nabla_x f(\x_{t}, \y_{t}; \xi_{i_t,t}, i_t)] = (\nabla_x f_1(\x_{t}, \y_{1,t}), \ldots, \nabla_x f_m(\x_t, \y_{m,t})),\\
&\E_{\xi, i_t}[\widetilde\nabla_x f(\x_{t-1}, \y_{t-1}; \xi_{i_t,t}, i_t)] = (\nabla_x f_1(\x_{t-1}, \y_{1,t-1}), \ldots, \nabla_x f_m(\x_{t-1}, \y_{m,t-1})). 
\end{align*}
This step is applying the STORM technique for updating $\u_{t+1}$ based on randomized stochastic unbiased estimators, where only one coordinate of the unbiased estimators is non-zero. We refer to this step as randomized-coordinate STORM  (RC-STORM) update.

The RC-STORM update is the key to achieving efficiency, i.e., it only processes (non-trivially) one coordinate of  $\u_{t+1}$  for a sampled lower problem $i_t$ explicitly, where $i_t$ is sampled following a distribution $\Pr(i_t= i) = p_i$. 
For other coordinates $i\neq i_t$,  $\u_{i,t+1}$ are updated trivially, 
which are scaled by $(1-\beta_{fx, t})$. This scaling can be delayed  until it is sampled in the future since it is not used in updating $\x_{t+1}$  presumably unless $j_t = i$. In this case, there will be one  more coordinate of $\u_t$ that will be evaluated at most twice for the current and the last iteration. Overall, at each iteration, we only sample data from one lower problem to update the corresponding $\u_{i,t+1}$ and evaluate trivially another coordinate at most twice without sampling data from the corresponding lower problem. Similar updates are applied to $\v_{t+1}, \w_{t+1}, V_{t+1}, H_{t+1}$.

Next, we present the convergence of Algorithm~\ref{alg:2} below. Define $C_{fx}(\p) = \sum_{i=1}^m C^2_{fx, i}/p_i$ and $L_{fx}(\p) = \sum_{i=1}^mL^2_{fx,i}/p_i$, and similarly for other quantities. Let $C(\p) = C_{fx}(\p)+C_{fy}(\p)+C_{gy}(\p)+C_{gxy}(\p)+C_{gyy}(\p)$ and $L(\p) = L_{fx}(\p) + L_{fy}(\p) + L_{gy}(\p) + L_{gxy}(\p) + L_{gyy}(\p)$. 
\begin{thm}\label{thm:2}
Suppose Assumptions~\ref{ass:3} and~\ref{ass:4} hold. 
With $\tau\leq 1/(3L_g)$, $\gamma \leq \min\{\frac{\sqrt{\tau\lambda m}}{8\sqrt{C}L_y}, \frac{m}{24L(\p)}\}$, $\eta_t=\tau_t=\eta= \min \{\frac{1}{\sqrt{3(C_1+C_2+C_3+C_4+C_{\max}') \gamma}}, \frac{1}{2L_F  \gamma},\frac{\epsilon}{2}\sqrt{\frac{m}{\gamma C(\p) C_5}}\}$, $C\geq\max\{\frac{6L(\p)}{\gamma \tau},12C_0 \tau \lambda\}$, $\beta_{fx, t+1} =  3C_1 \gamma \eta_t^2$, 
$\beta_{fy, t+1} =  3C_4 \gamma \eta_t^2$, 
$\beta_{gxy, t+1} =  3C_2 \gamma \eta_t^2$, 
$\beta_{gyy, t+1} =  3C_3 \gamma \eta_t^2$, 
and $\beta_{t+1} = 2\gamma C'_{max} \eta_t^2$, where $C_1, C_2, C_3, C_4, C_5, C_{max}'$ are appropriate constants  specified in the proof, and by using a large mini-batch size of $O(1/m\epsilon)$ at the initial iteration for computing $\u_0, \v_0, H_0, V_0, \w_0$ and computing an accurate solution $\y_0$ such that $\|\y_0 - \y^*(\x_0)\|\leq O(L(\p))$. Then, in order to have $\E\bigg[\frac{1}{T+1}\sum_{t=0}^T\|\nabla F(\x_t)\|^2]\bigg]\leq \epsilon^2$, we need a total sample complexity
$T =  O(C(\p)/(m\epsilon^3))$. 
\end{thm}
{\bf Remark:} The above theorem implies that in order to find an epsilon stationary point, we need $ O(C(\p)/(m\epsilon^3))$ sample complexity, which is no worse than $O(m/\epsilon^3)$, and could be in a lower order if we could optimize over the sampling probability $\p$. The proof is included in the Appendix~\ref{sec:thm2}. 
\vspace{-0.1in}
\section{Faster Convergence for Gradient-Dominant Functions}
In this section, we present faster algorithms and their convergence for problems satisfying a stronger condition, namely the gradient dominant condition  or Polyak-\L ojasiewicz (PL) condition, i.e., 
\begin{align*}
\mu(F(\x) - \min_{\x'} F(\x'))\leq \|\nabla F(\x)\|^2. 
\end{align*}
To leverage the above condition for achieving a faster rate, we employ the restarting trick with a stagewise decreasing step size strategy. The procedure is described in~Algorithm \ref{alg:RECOVER}.

\begin{thm}\label{thm:main}
Suppose Assumptions~\ref{ass:3} and~\ref{ass:4} hold and the PL condition holds. Define a constant $\epsilon_1=\frac{5\eta_0^2\gamma C(\p)C_5}{C_6 m\mu}$ and $\epsilon_k=\epsilon_1/2^{k-1}$. By setting $\eta_1=\min\{\frac{1}{\sqrt{3(C_1+C_2+C_3+C_4+C'_{max})}},\frac{1}{2L_F\gamma}\}$, $T_1=O(\max\{\frac{1}{\mu\eta_1\gamma},\frac{1}{\mu\eta_1^2\gamma}\})$, and for $k\geq2$, $\eta_k = \sqrt{\frac{C_6m\mu\epsilon_k}{5\gamma C(\p)C_5}}$, $T_k = O(\max \{\frac{1}{\mu \eta_k \gamma}, \frac{1}{ \eta_k^2 \gamma}\})$, where and $C_0\sim C_5$, $\gamma$, $C_{max}'$ are as used in Theorem \ref{thm:2} and  $C_6=\min\{2,C_0,C_1,C_2,C_3,C_4\}$, then after  $K=O(\log(\epsilon_1/\epsilon))$ stages, the output of RE-RSVRG satisfies  $\E[F(\x_K) - F(\x_*)]\leq \epsilon$. 
\end{thm}
{\bf Remark:} It follows that the total sample complexity is $\sum_{k=1}^{K} T_k = O(\frac{C(\p)}{\mu^{3/2} m \epsilon^{1/2}} + \frac{C(\p)}{\mu m \epsilon})$. Similar to remark for Theorem \ref{thm:2}, the sample complexity is no worse than $O(\frac{m}{\mu^{3/2} \epsilon^{1/2}} + \frac{m}{\mu \epsilon})$. If $\mu\geq \epsilon$, the dependence on $\mu$ and $\epsilon$ matches the optimal rate of $O(\frac{1}{\mu \epsilon})$ to minimize a strongly convex problem, which is a stronger condition than PL condition. 
In contrast, to get $\E\|\x-\x_*\|^2\leq \epsilon$, the STABLE algorithm \cite{DBLP:journals/corr/abs-2102-04671} takes a complexity of $O(\frac{m}{\mu^4 \epsilon})$, TTSA  \cite{{DBLP:journals/corr/abs-2007-05170}} takes  $\widetilde{O}(\frac{m}{\mu^3\epsilon^{3/2}})$, BSA \cite{99401} takes $O(\frac{m}{\mu^5\epsilon^2})$, all  under the strong convexity, while ours only takes $O( \frac{m}{\mu^2 \epsilon})$ to achieve this under the PL condition.  

\begin{algorithm}[t]
    \centering
    \caption{RE-RSVRB($\w_0, \epsilon_{0} ,c$)}
    \label{alg:RECOVER}
    \begin{algorithmic}[1]
  \STATE Let $\Theta_0 = (\x_{0}, \y_{0}, \u_{0}, \v_{0}, V_{0}, H_{0}, \w_{0}, \d_{0}, \z_{0})$
       \FOR {$k = 1,\ldots,K$} 
        \STATE \hspace*{-0.1in}
        $\Theta_k$ = RSVRB$(\Theta_{k-1}, \eta_k, \tau_k, T_k)$
       \STATE \hspace*{-0.1in}change $\eta_k, \tau_k,  T_k$ according to Theorem~\ref{thm:main}
        \ENDFOR
        \STATE \textbf{Return:} $\x_K$
    \end{algorithmic}
\end{algorithm}

\section{Proof Sketch}
Here we provide a proof sketch of the proposed algorithms and highlight the key differences from \cite{DBLP:journals/corr/abs-2102-04671}. We start with Lemma \ref{lem:01} which follows standard analysis for non-convex optimization. 

\begin{lemma}\label{lem:01}
Let $\x_{t+1} = \x_t - \eta_t\gamma\d_{t+1}$. For $\gamma\eta_tL_F\leq 1/2$, we have
\begin{align*}
F(\x_{t+1})\leq F(\x_t) +   \frac{ \gamma\eta_t}{2}\|\nabla F(\x_t) - \d_{t+1}\|^2- \frac{\gamma\eta_t}{2}\|\nabla F(\x_t)\|^2  - \frac{\gamma\eta_t}{4}\|\d_{t+1}\|^2.
\end{align*}
\end{lemma}
The problem then reduces to bounding the error for the stochastic estimator $\d_{t+1}$. Then, we decompose this error into several terms and bound them separately. 
\begin{lemma}\label{lem:02}
For all $t\geq 0$, we have
\begin{align*}
&\|\d_{t+1} - \nabla F(\x_t)\|^2\leq 2\|\d_{t+1} - \frac{1}{m}\sum_{i=1}^m\z_{i,t+1}\|^2+\frac{1}{m}\sum_{i=1}^m4C_0\|\y_{i,t} - \y^*(\x_{t})\|^2 \\
&+  \frac{1}{m}\sum_{i=1}^m(4C_1\|\u_{i,t+1} -  \nabla_x f_i(\x_t, \y_{i,t}) \|^2 + 4C_2\|V_{i,t+1} - \nabla^2_{xy}g_i(\x_t, \y_{i,t})\|^2 ) \\
& +\frac{1}{m}\sum_{i=1}^m (4C_3\|H_{i,t+1} -   \nabla_{yy}^2g_i(\x_t, \y_{i,t})\|^2 +4C_4\|\v_{i,t+1} -\nabla_y f_i(\x_t, \y_{i,t}) \|^2)
\end{align*}
where $C_0, C_1, C_2, C_3, C_4$ are constants.
\end{lemma}
Next, we will bound each term in the RHS in the above inequality separately. For the last four terms, we use the analysis of STORM to bound the error. For the first term, we use Lemma~\ref{lem:11} and Lemma~\ref{lem:12}. For bounding these terms, we will have a term in the upper bound that is proportional to $\|\x_{t+1} - \x_t\|^2$ and $\|\y_{t+1} - \y_t\|^2$, where $\y_t= (\y_{1,t}, \ldots, \y_{m,t})$. In~\citep{DBLP:journals/corr/abs-2102-04671}, the authors bound these two terms by their step sizes in the order of $\epsilon^2$, which is  in the order of the required accuracy bound for $\|\nabla F(\x)\|^2$.  However, the differences of our algorithms are that (i) the gradient estimators are not necessarily bounded (Algorithm~\ref{alg:1}) due to weaker boundness assumption in Assumption~\ref{ass:2}; (ii) even with similar boundness assumption in Assumption~\ref{ass:4}, our step size is in the order of $\epsilon$ (using fixed step size), which is one order larger than the required accuracy bound. Hence, using the step size to bound $\|\x_{t+1} - \x_t\|^2$ and $\|\y_{t+1} - \y_t\|^2$ as in~\citep{DBLP:journals/corr/abs-2102-04671} will not yield an improved sample complexity. To this end, we will use the following lemma to bound $\|\y_t - \y^*(\x_t)\|^2$ due to~\citep{DBLP:journals/corr/abs-2008-08170}. 

\begin{lemma}\label{lem:03}
Let $\y_{t+1} =(1-\tau_t)\y_t +  \tau_t\Pi_{\Y}[\y_t - \tau\w_{t+1}]$ with $\tau\leq 2/(3L_g)$ and $g(\x, \y)$ satisfy Assumption~\ref{ass:1}, we have
\begin{align*}
&\|\y_{t+1} - \y^*(\x_{t+1})\|^2 \leq  (1 - \frac{\tau_t\tau \lambda}{4})\|\y_t - \y^*(\x_t)\|^2 + \frac{8\tau_t\tau}{\lambda}\|\nabla_y g(\x_t, \y_t) -\w_{t+1}\|^2 \\
& + \frac{8L_y^2\gamma^2\eta_t^2}{\tau_t\tau\lambda}\|\z_{t+1}\|^2  -  \frac{2\tau}{\tau_t}(1+ \frac{\tau_t\tau \lambda}{4})(\frac{1}{2\tau} - \frac{3L_g}{4} )\|\y_t- \y_{t+1}\|^2.
\end{align*}
\end{lemma}
It is notable that the upper bound in this lemma has a negative term $\|\y_t - \y_{t+1}\|^2$, which will cancel the $\|\y_t- \y_{t+1}\|^2$ from bounding the first and the last four terms in the RHS of the inequality in Lemma~\ref{lem:02}. Finally, with these lemmas, we can prove Theorem~\ref{thm:2} with appropriate parameter setting.

\section{Conclusions}
In this paper, we have proposed a randomized stochastic algorithm for solving SBO with many lower level problems. We used a recursive variance reduction method for estimating all first-order and second-order moments. We achieved the state-of-the-art sample complexity for solving a class of SBO problems. One drawback of the proposed algorithms is that they need to compute the projection of a matrix onto a positive-definite matrix space with a minimum eigen-value, which could have a cubic time complexity in the worst case.  For future work, we will consider reducing the worst-case  per-iteration costs further. 

\bibliography{ref,example_paper,all}

\newpage
\required{ \hspace*{0.8in}\Large Appendix}
\appendix

\section{Convergence Analysis of Theorem~\ref{thm:1}}\label{sec:thm1}
In this section, we present the convergence analysis of Theorem~\ref{thm:1}. To this end, we will first present several technical lemmas. 

\begin{lemma}\label{lem:1}
Let $\x_{t+1} = \x_t - \eta_t\gamma\z_{t+1}$. For $\gamma\eta_tL_F\leq 1/2$, we have
\begin{align*}
F(\x_{t+1})\leq F(\x_t) +   \frac{ \gamma\eta_t}{2}\|\nabla F(\x_t) - \z_{t+1}\|^2- \frac{\gamma\eta_t}{2}\|\nabla F(\x_t)\|^2  - \frac{\gamma\eta_t}{4}\|\z_{t+1}\|^2.
\end{align*}
\end{lemma}
From the above lemma, we can see that the key to the proof of Theorem~\ref{thm:1} is the bounding of $\|\z_{t+1} - \nabla F(\x_t)\|^2$. The lemma below will decompose this error into several terms that can be bounded separately. 
\begin{lemma}\label{lem:2}
For all $t\geq 0$, we have
\begin{align*}
&\|\z_{t+1} - \nabla F(\x_t)\|^2 = \|\u_{t+1} -V_{t+1}[H_{t+1}]^{-1}\v_{t+1}  - \nabla F(\x_t)\|^2\\
&\leq 2C_0\|\y_t - \y^*(\x_t)\|^2 + 2C_1\|\u_{t+1} -  \nabla_x f(\x_t, \y_t) \|^2 \\
& + 2C_2\|V_{t+1} - \nabla^2_{xy}g(\x_t, \y_t)\|_F^2 +2C_3\|H_{t+1} - \nabla_{yy}^2g(\x_t, \y_t)\|_F^2 +2C_4\|\v_{t+1} -\nabla_y f(\x_t, \y_t) \|^2.
\end{align*}
where $C_0=(2L_{fx}^2 +  \frac{6C_{fy}^2L_{gxy}^2}{\lambda^2} + \frac{6C_{fy}^2C^2_{gxy}L^2_{gyy}}{\lambda^4}+ \frac{6L_{fy}^2C^2_{gxy}}{\lambda^2})$, 
 $C_1 = 2$, $C_2 = \frac{6C_{fy}^2}{\lambda^2}$, $C_3 = \frac{6C_{fy}^2C_{gxy}^2}{\lambda^4}$ and $C_4 = \frac{6C_{gxy}^2}{\lambda^2}$. 
\end{lemma}

Next, we will bound each term in the RHS in the inequality of the above lemma separately. We first bound the first term. 

\begin{lemma}\label{lem:3}
Let $\y_{t+1} =\y_t - \tau_t\tau \w_{t+1}$ with $\tau\leq 2/(3L_g)$, we have
\begin{align*}
&\|\y_{t+1} - \y^*(\x_{t+1})\|^2 \leq  (1 - \frac{\tau_t\tau \lambda}{4})\|\y_t - \y^*(\x_t)\|^2 + \frac{8\tau_t\tau}{\lambda}\|\nabla_y g(\x_t, \y_t) -\w_{t+1}\|^2 \\
& + \frac{8L_y^2\gamma^2\eta_t^2}{\tau_t\tau\lambda}\|\z_{t+1}\|^2  -  \frac{2\tau}{\tau_t}(1+ \frac{\tau_t\tau \lambda}{4})(\frac{1}{2\tau} - \frac{3L_g}{4} )\|\y_t- \y_{t+1}\|^2. 
\end{align*}
\end{lemma}
Based on this lemma, we derive the following lemma. 
\begin{lemma}\label{lem:6}
Let $\delta_{y,t}=\| \y_{t+1} - \y^*(\x_t)\|^2$, we have
\begin{align*}
&\bigg[\sum_{t=0}^T\E(\delta_{y, t+1} - \delta_{y,t})\bigg]\leq \sum_{t=0}^T\frac{8\tau_t\tau}{\lambda}\E\|\nabla_y g(\x_t, \y_t) -\w_{t+1}\|^2  + \sum_{t=0}^T(- \frac{\tau_t\tau\lambda}{4}) \E[\delta_{y,t}]\\
& + \sum_{t=0}^T\frac{8L_y^2\gamma^2\eta_t^2}{\tau_t\tau\lambda}\E\|\z_{t+1}\|^2  - \sum_{t=0}^T  \frac{2\tau}{\tau_t}(1+ \frac{\tau_t\tau \lambda}{4})(\frac{1}{2\tau} - \frac{3L_g}{4} )\E\|\y_t- \y_{t+1}\|^2. 
\end{align*}
\end{lemma}
For the last four terms in the RHS of the inequality in Lemma~\ref{lem:2}, we can use the following lemma about the variance reduced property of the STROM update. 
\begin{lemma}\label{lem:4}
Suppose $h: \R^d \rightarrow \R^{d'}$ is a $L$-Lipschitz continuous mapping, $\E_{\xi}[h(\e; \xi)] = h(\e)$, and $\E[\|h(\e; \xi) - h(\e)\|^2]\leq \sigma^2$. For any sequence $\{\e_0, \e_1, \ldots, \e_T\}$ where $\e_t\in\R^d$,  let $\hat h_{t+1} = (1-\beta_t)(\hat h_t  - h(\e_{t-1}; \xi_t)) + h(\e_t; \xi_t), t\geq 0$, for an Euclidean norm $\|\cdot\|$,
\begin{align*}
\E_t[\|\hat h_{t+1} - h(\e_t)\|^2]\leq (1-\beta_t) \|\hat h_t - h(\e_{t-1})\|^2 + 2\beta_t^2 \sigma^2 +2 L^2\|\e_{t} - \e_{t-1}\|^2.
\end{align*}
If $h(\e)\in\Omega$ for a convex set $\Omega$, let $\hat h_{t+1} = \prod_{\Omega}[(1-\beta_t)(\hat h_t  - h(\e_{t-1}; \xi_t)) + h(\e_t; \xi_t)], t\geq 0$
\begin{align*}
\E_t[\|\hat h_{t+1} - h(\e_t)\|^2]\leq (1-\beta_t) \|\hat h_t - h(\e_{t-1})\|^2 + 2\beta_t^2 \sigma^2 + 2L^2\|\e_{t} - \e_{t-1}\|^2.
\end{align*}
\end{lemma}
{\bf Remark:} The above lemma can be easily proved by following the analysis in~\citep{cutkosky2019momentum}. 
Based on this lemma we have the following lemma. 
\begin{lemma}\label{lem:5}
Let $\delta_{t}=\|\hat h_{t+1} - h(\e_t)\|^2$, we have
\begin{align*}
\sum_{t=0}^T\E(\frac{\delta_{t+1}}{\eta_t} - \frac{\delta_{t}}{\eta_{t-1}})\leq \sum_{t=0}^T\frac{2\beta_{t+1}^2\sigma^2}{\eta_t}  + \sum_{t=0}^T(\frac{1}{\eta_t} - \frac{1}{\eta_{t-1}} - \frac{\beta_{t+1}}{\eta_t}) \E[\delta_t] +  \sum_{t=0}^T\frac{2L^2}{\eta_t}\E\|\e_{t+1} - \e_t\|^2.
\end{align*}
\end{lemma}

\begin{proof}[Proof of Theorem~\ref{thm:1}]
First, we apply Lemma~\ref{lem:5} to $\delta_{gy, t}= \|\w_{t+1} - \nabla_y g(\x_t, \y_t)\|^2$. We have
\begin{align*}
&\E\bigg[\sum_{t=1}^T\frac{\delta_{gy, t+1}}{\eta_t} - \frac{\delta_{gy, t}}{\eta_{t-1}}\bigg]\leq \sum_{t=1}^T\frac{2\beta_{gy, t+1}^2\sigma^2}{\eta_t}  + \sum_{t=1}^T(\frac{1}{\eta_t} - \frac{1}{\eta_{t-1}} - \frac{\beta_{gy, t+1}}{\eta_t}) \E[\delta_{gy, t}] \\
&+  \sum_{t=1}^T\frac{2L_{gy}^2}{\eta_t}\E(\|\x_{t+1} - \x_t\|^2 + \|\y_{t+1} - \y_t\|^2).
\end{align*}
Set $\eta_t = \tau_t = c/(c_0 + t)^{1/3}$.
To ensure $\eta_t \leq \frac{1}{4L_F}$, we need $c_0\geq (4 L_F c)^3$.
Thus,
\begin{equation}
\begin{split}
\frac{1}{\eta_t} - \frac{1}{\eta_{t-1}} &= \frac{(c_0+t)^{1/3}}{c} - \frac{(c_0+t-1)^{1/3}}{c} \\
&\leq \frac{1}{3c(c_0+t-1)^{2/3}} \leq  \frac{1}{3c(c_0/2+t)^{2/3}}  \\
&\leq \frac{2^{2/3}}{3c(c_0+t)^{2/3}} \leq \frac{2^{2/3}}{3c^3} \eta_t^2 \leq \frac{1}{7L_F c^3} \eta_t, \\
\end{split}
\end{equation}
where the first inequality holds by the concavity of the function $f(x) = x^{1/3}$, i.e., $(x+y)^{1/3} \leq x^{1/3} + \frac{y}{3x^{2/3}}$, the second inequality is because $c_0 \geq 2$.
With $\beta_{gy, t+1} = (\frac{1}{7L_Fc^3} + \frac{8C\gamma\tau_t\tau}{\lambda\eta_t} )\eta_t^2$ ($\beta_{gy, t+1}< 1$ for large enough constant $c_0\geq \max\{\left(\frac{2}{7L_Fc}\right)^{3/2}, \left(\frac{c^2\lambda}{16C\gamma\tau}\right)^{3/2}\}$. By combining the recursion of $\delta_{y, t+1}$ with $\delta_{gy, t+1}$, we have 
\begin{align*}
&\E\bigg[\sum_{t=0}^TC\gamma(\delta_{y,t+1} - \delta_{y,t})\bigg] + \E\bigg[\sum_{t=0}^T(\frac{\delta_{gy, t+1}}{\eta_t} - \frac{\delta_{gy, t}}{\eta_{t-1}})\bigg]\\
&\leq \sum_{t=0}^T\frac{2\beta_{gy, t+1}^2\sigma^2}{\eta_t} +  \sum_{t=0}^T(-  \frac{C\gamma\tau_t\tau\lambda}{4}) \E[\delta_{y,t}]\\
& + \sum_{t=0}^T\frac{8C\gamma^3L_y^2\eta_t^2}{\tau_t\tau\lambda}\E\|\z_{t+1}\|^2  - \sum_{t=0}^T \frac{2C\gamma\tau}{\tau_t}(1+ \frac{\tau_t\tau \lambda}{4})(\frac{1}{2\tau} - \frac{3L_g}{4} )\E\|\y_t- \y_{t+1}\|^2\\ 
&+  \sum_{t=0}^T\frac{2L_{gy}^2}{\eta_t}\E(\|\x_{t+1} - \x_t\|^2 + \|\y_{t+1} - \y_t\|^2),
\end{align*}
where $C$ is given below. 

To continue, we add the above recursion with the recursions for $\delta_{fx,t}=\|\u_{t+1} - \nabla_x f(\x_t, \y_t)\|^2$, $\delta_{fy,t}=\|\v_{t+1} - \nabla_y f(\x_t; \y_t)\|^2$, $\delta_{gxy, t}= \|V_{t+1} - \nabla_{xy}^2 g(\x_t, \y_t)\|_F^2$, $\delta_{gyy,t} = \|H_{t+1} - \nabla_{yy}^2 g(\x_t, \y_t)\|^2$, we have
\begin{align*}
&\E\bigg[\sum_{t=0}^TC\gamma(\delta_{y,t+1} - \delta_{y,t})\bigg] + \E\bigg[\sum_{t=0}^T(\frac{\delta_{gy, t+1}}{\eta_t} - \frac{\delta_{gy, t}}{\eta_{t-1}})\bigg] + \E\bigg[\sum_{t=0}^T(\frac{\delta_{fx, t+1}}{\eta_t} - \frac{\delta_{fx, t}}{\eta_{t-1}})\bigg]\\
& +  \E\bigg[\sum_{t=0}^T(\frac{\delta_{fy, t+1}}{\eta_t} - \frac{\delta_{fy, t}}{\eta_{t-1}})\bigg]+ \E\bigg[\sum_{t=0}^T(\frac{\delta_{gxy, t+1}}{\eta_t} - \frac{\delta_{gxy, t}}{\eta_{t-1}})\bigg]  + \E\bigg[\sum_{t=0}^T(\frac{\delta_{gyy, t+1}}{\eta_t} - \frac{\delta_{gyy, t}}{\eta_{t-1}})\bigg]\\
&\leq \sum_{t=0}^T\frac{2(\beta_{gy, t+1}^2+ \beta_{fx,t+1}^2 + \beta_{fy,t+1}^2 + \beta_{gxy, t+1}^2 + \beta_{gyy,t+1}^2)\sigma^2}{\eta_t} +  \sum_{t=0}^T(- \frac{C\gamma\tau_t\tau\lambda}{4}) \E[\delta_{y,t}]\\
& ~~~+ \sum_{t=0}^T(\frac{1}{\eta_t} - \frac{1}{\eta_{t-1}} - \frac{\beta_{fx,t+1}}{\eta_t}) \E[\delta_{fx,t}] +  \sum_{t=0}^T(\frac{1}{\eta_t} - \frac{1}{\eta_{t-1}} - \frac{\beta_{fy, t+1}}{\eta_t}) \E[\delta_{fy,t}]  \\
&~~~+ \sum_{t=0}^T(\frac{1}{\eta_t} - \frac{1}{\eta_{t-1}} - \frac{\beta_{gxy,t+1}}{\eta_t}) \E[\delta_{gxy,t}] +  \sum_{t=0}^T(\frac{1}{\eta_t} - \frac{1}{\eta_{t-1}} - \frac{\beta_{gyy, t+1}}{\eta_t}) \E[\delta_{gyy,t}]  \\
&~~~ + \sum_{t=0}^T\frac{8C\gamma^3\eta_t^2L_y^2}{\tau_t\tau\lambda}\E\|\z_{t+1}\|^2  - \sum_{t=1}^T \frac{2C\gamma\tau}{\tau_t}(1+ \frac{\tau_t\tau \lambda}{4})(\frac{1}{2\tau} - \frac{3L_g}{4} )\E\|\y_t- \y_{t+1}\|^2\\
&~~~ +  \sum_{t=0}^T\frac{2(L_{gy}^2+L_{fx}^2 + L_{fy}^2 + L_{gxy}^2 + L_{gyy}^2)}{\eta_t}\E(\|\x_{t+1} - \x_t\|^2 + \|\y_{t+1} - \y_t\|^2).
\end{align*}
With large enough constant $c_0\geq \max\{\left(\frac{2}{7L_Fc}\right)^{3/2},(2C_1\gamma c^2)^{3/2},(2C_2\gamma c^2)^{3/2},(2C_3\gamma c^2)^{3/2},(2C_4\gamma c^2)^{3/2}\}$, by setting $\beta_{fx, t+1} = (\frac{1}{7L_Fc^3}  + C_1\gamma)\eta_t^2<1$, $\beta_{fy, t+1} = (\frac{1}{7L_Fc^3}  + C_4\gamma)\eta_t^2<1$,  $\beta_{gxy, t+1} = (\frac{1}{7L_Fc^3}  + C_2\gamma)\eta_t^2<1$, $\beta_{gyy, t+1} = (\frac{1}{7L_Fc^3}  + C_3\gamma)\eta_t^2<1$, $ \frac{C\tau_t\tau\lambda}{4} \geq C_0\eta_t$ (which means the constant $C\geq \frac{4C_0}{\tau\lambda}$), we have 
\begin{align*}
&\E\bigg[\sum_{t=0}^TC\gamma(\delta_{y,t+1} - \delta_{y,t})\bigg] + \E\bigg[\sum_{t=0}^T(\frac{\delta_{gy, t+1}}{\eta_t} - \frac{\delta_{gy, t}}{\eta_{t-1}})\bigg] + 
\E\left[\sum\limits_{t=0}^{T} (\frac{\delta_{fx,t+1}}{\eta_t} - \frac{\delta_{fx,t}}{\eta_{t-1}}) \right] \\ 
&+ \E\bigg[\sum_{t=0}^T(\frac{\delta_{fy, t+1}}{\eta_t} - \frac{\delta_{fy, t}}{\eta_{t-1}})\bigg]
+ \E\bigg[\sum_{t=0}^T(\frac{\delta_{gxy, t+1}}{\eta_t} - \frac{\delta_{gxy, t}}{\eta_{t-1}})\bigg]  + \E\bigg[\sum_{t=0}^T(\frac{\delta_{gyy, t+1}}{\eta_t} - \frac{\delta_{gyy, t}}{\eta_{t-1}})\bigg]\\
&\leq \sum_{t=0}^T\frac{2(\beta_{gy, t+1}^2+ \beta_{fx,t+1}^2 + \beta_{fy,t+1}^2 + \beta_{gxy, t+1}^2 + \beta_{gyy,t+1}^2)\sigma^2}{\eta_t}\\
& -  \sum_{t=0}^TC_0\gamma \eta_t\E[\delta_{y,t}]-  \sum_{t=0}^TC_1\gamma\eta_t\E[\delta_{fx,t}] -  \sum_{t=0}^TC_4\gamma\eta_t \E[\delta_{fy,t}]-  \sum_{t=0}^TC_2\gamma\eta_t \E[\delta_{gxy,t}] -   \sum_{t=0}^TC_3\gamma\eta_t \E[\delta_{gyy,t}]  \\
& + \sum_{t=0}^T(\frac{8C\gamma L_y^2}{\tau_t\tau\lambda} + \frac{2(L_{gy}^2+L_{fx}^2 + L_{fy}^2 + L_{gxy}^2 + L_{gyy}^2)}{\eta_t})\eta_t^2\gamma^2\E[\|\z_{t+1}\|^2] \\
& - \sum_{t=0}^T \frac{2C\gamma\tau}{\tau_t}(1+ \frac{\tau_t\tau \lambda}{4})(\frac{1}{2\tau} - \frac{3L_g}{4} )\E[\|\y_t- \y_{t+1}\|^2]\\
&+  \sum_{t=0}^T\frac{2(L_{gy}^2+L_{fx}^2 + L_{fy}^2 + L_{gxy}^2 + L_{gyy}^2)}{\eta_t}(\E[\|\y_{t+1} - \y_t\|^2]).
\end{align*}
With $\tau\leq 1/{(3L_g)}$ and $C \geq  \frac{4(L^2_{gy} + L^2_{fx} + L^2_{fy} + L^2_{gxy} + L^2_{gyy})}{\gamma}$, we have $\frac{1}{2\tau} - \frac{3L_g}{4}\geq \frac{1}{4\tau}$, $\frac{C\gamma}{2\tau_t}\geq \frac{2(L_{gy}^2+L_{fx}^2 +  L_{fy}^2 + L_{gxy}^2 + L_{gyy}^2)}{\eta_t}$. As a result, 
\begin{align*}
&\E\bigg[\sum_{t=0}^TC\gamma(\delta_{y,t+1} - \delta_{y,t})\bigg] + \E\bigg[\sum_{t=0}^T(\frac{\delta_{gy, t+1}}{\eta_t} - \frac{\delta_{gy, t}}{\eta_{t-1}})\bigg] + 
\E\left[\sum\limits_{t=0}^{T} (\frac{\delta_{fx,t+1}}{\eta_t} - \frac{\delta_{fx,t}}{\eta_{t-1}}) \right] \\ 
&+ \E\bigg[\sum_{t=0}^T(\frac{\delta_{fy, t+1}}{\eta_t} - \frac{\delta_{fy, t}}{\eta_{t-1}})\bigg]
+ \E\bigg[\sum_{t=0}^T(\frac{\delta_{gxy, t+1}}{\eta_t} - \frac{\delta_{gxy, t}}{\eta_{t-1}})\bigg]  +  \E\bigg[\sum_{t=0}^T(\frac{\delta_{gyy, t+1}}{\eta_t} - \frac{\delta_{gyy, t}}{\eta_{t-1}})\bigg]\\
&\leq \sum_{t=0}^T\frac{2(\beta_{gy, t+1}^2+ \beta_{fx,t+1}^2 + \beta_{fy,t+1}^2 + \beta_{gxy, t+1}^2 + \beta_{gyy,t+1}^2)\sigma^2}{\eta_t}\\
& - \sum_{t=0}^TC_0 \gamma\eta_t\E[\delta_{y,t}]-  \sum_{t=0}^TC_1\gamma\eta_t\E[\delta_{fx,t}] -  \sum_{t=0}^TC_4\gamma\eta_t \E[\delta_{fy,t}]-  \sum_{t=0}^TC_2\gamma\eta_t \E[\delta_{gxy,t}] -   \sum_{t=0}^TC_3\gamma\eta_t \E[\delta_{gyy,t}]  \\
& + \sum_{t=0}^T(\frac{8C \gamma L_y^2}{\tau_t \tau \lambda} + \frac{2(L_{gy}^2+L_{fx}^2 + L_{fy}^2 + L_{gxy}^2 + L_{gyy}^2)}{\eta_t})\gamma^2\eta_t^2\E[\|\z_{t+1}\|^2].   
\end{align*} 
Adding the above inequality with that in Lemma~\ref{lem:1}, we have
\begin{align*}
&\E\bigg[\sum_{t=0}^T \frac{\gamma\eta_t}{2}\|\nabla F(\x_t)\|^2 \bigg]+ \E\bigg[\sum_{t=0}^TC\gamma (\delta_{y,t+1} - \delta_{y,t})\bigg] + \E\bigg[\sum_{t=0}^T(\frac{\delta_{gy, t+1}}{\eta_t} - \frac{\delta_{gy, t}}{\eta_{t-1}})\bigg] + \E\bigg[\sum_{t=0}^T(\frac{\delta_{fy, t+1}}{\eta_t} - \frac{\delta_{fy, t}}{\eta_{t-1}})\bigg]\\
&+\E\bigg[\sum_{t=0}^T(\frac{\delta_{fx, t+1}}{\eta_t} - \frac{\delta_{fx, t}}{\eta_{t-1}})\bigg]+ \E\bigg[\sum_{t=0}^T(\frac{\delta_{gxy, t+1}}{\eta_t} - \frac{\delta_{gxy, t}}{\eta_{t-1}})\bigg]  + \E\bigg[\sum_{t=0}^T(\frac{\delta_{gyy, t+1}}{\eta_t} - \frac{\delta_{gyy, t}}{\eta_{t-1}})\bigg] \\ 
&  \leq \E\bigg[F(\x_0)  - F(\x_*)  -\sum_{t=0}^T \frac{\gamma\eta_t}{4}\|\z_{t+1}\|^2\bigg]\\
&  +\E\bigg[ \sum_{t=0}^T(\frac{8 C \gamma L_y^2}{\tau \lambda} + 2(L_{gy}^2+L_{fx}^2 + L_{fy}^2 + L_{gxy}^2 + L_{gyy}^2))\gamma^2\eta_t\|\z_{t+1}\|^2 \bigg] +  O(1)\sum_{t=0}^T\eta_t^3. 
\end{align*}
With $\gamma (\frac{8 C \gamma L_y^2}{\tau \lambda} + 2(L_{gy}^2+L_{fx}^2 + L_{fy}^2 + L_{gxy}^2 + L_{gyy}^2))\leq 1/4$ ($\gamma \leq \min\{\frac{\sqrt{\tau \lambda}}{8\sqrt{C}L_y}, \frac{1}{16(L_{gy}^2+L^2_{fx} + L^2_{fy} + L^2_{gxy} +  L^2_{gyy})}\}$), we have 
\begin{align*}
&\E\bigg[\sum_{t=0}^T \frac{\gamma\eta_t}{2}\|\nabla F(\x_t)\|^2\bigg] + \E\bigg[\sum_{t=0}^TC\gamma(\delta_{y,t+1} - \delta_{y,t})\bigg] + \E\bigg[\sum_{t=0}^T(\frac{\delta_{gy, t+1}}{\eta_t} - \frac{\delta_{gy, t}}{\eta_{t-1}})\bigg] + \E\bigg[\sum_{t=0}^T(\frac{\delta_{fy, t+1}}{\eta_t} - \frac{\delta_{fy, t}}{\eta_{t-1}})\bigg]\\
&+ \E\bigg[\sum_{t=0}^T(\frac{\delta_{gxy, t+1}}{\eta_t} - \frac{\delta_{gxy, t}}{\eta_{t-1}})\bigg]  + \E\bigg[\sum_{t=0}^T(\frac{\delta_{gyy, t+1}}{\eta_t} - \frac{\delta_{gyy, t}}{\eta_{t-1}})\bigg]\\
&  \leq F(\x_0)  - F(\x_*)   + O(\log (T+2)).
\end{align*}
As a result, 
\begin{align*}
\E\bigg[\sum_{t=0}^T \frac{\gamma\eta_t}{2}\|\nabla F(\x_t)\|^2\bigg] \leq \Delta + C\gamma \E[\delta_{y, 0}] + \frac{\E[\delta_{gy,0} +\delta_{fx,0} +  \delta_{fy,0} + \delta_{gxy,0} + \delta_{gyy,0}]}{\eta_0}  +O(\log (T+2)). 
\end{align*}
Thus, 
\begin{align*}
\E\bigg[\sum_{t=0}^T \frac{\gamma}{2(T+1)}\|\nabla F(\x_t)\|^2\bigg] \leq \frac{\Delta}{\eta_T T} + \frac{C\gamma \E[\delta_{y, 0}]}{\eta_T T} + \frac{\E[\delta_{gy,0} +\delta_{fx,0} +  \delta_{fy,0} + \delta_{gxy,0} + \delta_{gyy,0}]}{\eta_0 \eta_T T}  +\frac{O(\log (T+2))}{\eta_T T}. 
\end{align*}
According to the analysis in~\citep{cutkosky2019momentum}, we have
\begin{align*}
\E\bigg[\frac{1}{T+1}\sum_{t=0}^T\|\nabla F(\x_t)\|^2\bigg]\leq \widetilde O(\frac{1}{T^{2/3}}). 
\end{align*}
\end{proof}

\section{Convergence Analysis of Theorem~\ref{thm:2}}\label{sec:thm2}
In this section, we present the convergence analysis of Theorem~\ref{thm:2}. In the following, we let $\y_t = (\y_{1,t}, \ldots, \y_{m,t})$, $\u_t = (\u_{1,t}, \ldots, \u_{m,t})$, $\v_t = (\v_{1,t}, \ldots, \v_{m,t})$, $V_t = (V_{1,t}, \ldots, V_{m,t})$, $H_t = (H_{1,t}, \ldots, H_{m,t})$. By defining  $g(\x, \y)=\sum_i g_i(\x, \y_i)$,  we have $g(\x, \y)$ is $\lambda$ strongly convex in terms of $\y$. Hence Lemma~\ref{lem:3}, Lemma~\ref{lem:6} still hold. The Lemma~\ref{lem:1} also holds. The lemma below will decompose  $\|\d_{t+1} - \nabla F(\x_t)\|^2$ into several terms that can be bounded separately. 
\begin{lemma}\label{lem:2-2}
For all $t\geq 0$, we have
\begin{align*}
&\|\d_{t+1} - \nabla F(\x_t)\|^2\leq 2\|\d_{t+1} - \frac{1}{m}\sum_{i=1}^m\z_{i,t+1}\|^2+\frac{1}{m}\sum_{i=1}^m4C_0\|\y_{i,t} - \y^*(\x_{i,t})\|^2 \\
&+  \frac{1}{m}\sum_{i=1}^m(4C_1\|\u_{i,t+1} -  \nabla_x f_i(\x_t, \y_{i,t}) \|^2 +4C_2\|V_{i,t+1} - \nabla^2_{xy}g_i(\x_t, \y_{i,t})\|^2) \\
& +\frac{1}{m}\sum_{i=1}^m (4C_3\|H_{i,t+1} - \nabla_{yy}^2g_i(\x_t, \y_{i,t})\|^2  +4C_4\|\v_{i,t+1} -\nabla_y f_i(\x_t, \y_{i,t}) \|^2),
\end{align*}
where $C_0=\max\limits_{1\leq i \leq m} (2L_{fx,i}^2 +  \frac{6C_{fy,i}^2L_{gxy,i}^2}{\lambda^2} + \frac{6C_{fy,i}^2C^2_{gxy,i}L^2_{gyy,i}}{\lambda^4}+ \frac{6L_{fy,i}^2C^2_{gxy,i}}{\lambda^2})$, $C_1 = 2$, $C_2 =\max\limits_{1\leq i \leq m} ( \frac{6C_{fy,i}^2}{\lambda^2})$, $C_3 = \max\limits_{1\leq i \leq m} (\frac{6C_{fy,i}^2C_{gxy,i}^2}{\lambda^4})$ and $C_4 = \max\limits_{1\leq i \leq m} (\frac{6C_{gxy,i}^2}{\lambda^2})$. 
\end{lemma}
Next, we will bound each term in the RHS in the inequality of the above lemma separately. We first bound the first term. 

\begin{lemma} \label{lem:11}
Assume $\E_i[\|\z_i - \frac{1}{m}\sum_i\z_i \|^2]\leq \sigma_z^2$
\begin{align*}
&\E_{j_t}[\|\d_{t+1} - \frac{1}{m} \sum_{i=1}^m \z_{i,t+1}\|^2]\leq (1-\beta_t) \|\d_{t} - \frac{1}{m} \sum_{i=1}^m \z_{i,t}\|^2 + 2\beta_t^2\sigma^2 \\
& + C'_1\frac{1}{m}\sum_{i=1}^m[\|\u_{i, t+1} - \u_{i, t}\|^2] + C'_2\frac{1}{m}\sum_{i=1}^m[\|V_{i, t+1} - V_{i, t}\|^2] \\
&+  C'_3\frac{1}{m}\sum_{i=1}^m[\|H_{i, t+1} - H_{i, t}\|^2]   + C_4'\frac{1}{m}\sum_{i=1}^m[\|\v_{i, t+1} - \v_{i, t}\|^2]
\end{align*}
where $C_1' = 8$, $C_2'=\max\limits_{1\leq i\leq m}\frac{24C^2_{fy,i}}{\lambda}$, $C_3' = \max\limits_{1\leq i\leq m} \frac{24 C^2_{gxy,i}C^2_{fy,i}}{\lambda^4} $, $C_4'=\max\limits_{1\leq i\leq m} \frac{24C^2_{gxy,i}}{\lambda^2}$.  
\end{lemma}
The following lemma bounds the last four terms in RHS in the inequalities of the above two lemmas. 

\begin{lemma}\label{lem:12}Let $\u = (\u_1, \u_2, \ldots, \u_m)$. We apply RC-STROM update to $\u_{t+1}$ with unbiased stochastic estimators $\widetilde h(\e_t; \xi_t, i_t)= (0, \ldots, \frac{h_{i_t}(\e_t; \xi_t)}{p_{i_t}}, 0)$ and $\widetilde h(\e_{t-1}; \xi_t, i_t)= (0, \ldots, \frac{h_{i_t}(\e_{t-1}; \xi_t)}{p_{i_t}}, 0)$, i.e., $\u_{t+1} = \Pi_{\Omega}[(1-\beta_t)(\u_t - \widetilde h(\e_{t-1}; \xi_t, i_t)) + \widetilde h(\e_t; \xi_t, i_t)]$. Assume $\E[\|h_{i}(\e_t; \xi_t)\|^2]\leq C^2_i$ and $h_i$ is $L_i$-Lipschitz continuous. Define $C(\p) = \sum_{i=1}^m C^2_i/p_i$ and $L(\p) = \sum_{i=1}^mL^2_i/p_i$ and $\delta_t = \|\u_{t+1} - h(\e_t)\|^2$. Then we have
\begin{align*}
&\E_{\xi_t, i_t}[\|\u_{t+1}-\u_t\|^2]  \leq  4\beta_t^2\delta_{t-1} + 2\beta_t^2C(\p)+ 4(L(\p)+ \beta_t^2 \sum_iL^2_i)\|\e_t - \e_{t-1}\|^2
\end{align*}
and 
\begin{align*}
\E_{\xi_t, i_t}[\delta_t]\leq (1-\beta_t)\delta_{t-1} + 2\beta_t^2 C(\p) + 2L(\p)\|\e_{t} - \e_{t-1}\|^2.
\end{align*}
\end{lemma}

\begin{proof}[Proof of Theorem~\ref{thm:2}]  
Define $C_{fx}(\p)$, $C_{fy}(\p)$, $C_{gy}(\p)$,  $C_{gxy}(\p)$, $C_{gyy}(\p)$, $L_{fx}(\p)$,  $L_{fy}(\p)$, $L_{gy}(\p)$, $L_{gxy}(\p)$,  $L_{gyy}(\p)$ according to that in Lemma~\ref{lem:12} using Assumption~\ref{ass:4}. Let $L_{fx} =\sum_i L^2_{fx,i}$,  $L_{fy} =\sum_i L^2_{fy,i}$,  $L_{gy} =\sum_i L^2_{gy,i}$,  $L_{gxy} =\sum_i L^2_{gxy,i}$,  $L_{gyy} =\sum_i L^2_{gyy,i}$. We can see that $L_{*}\leq L_*(\p)$. Define $C(\p) = C_{fx}(\p)+C_{fy}(\p)+C_{gy}(\p)+C_{gxy}(\p)+C_{gyy}(\p)$ and $L(\p) = L_{fx}(\p) + L_{fy}(\p) + L_{gy}(\p) + L_{gxy}(\p) + L_{gyy}(\p)$.

Let $C'_{max} = 8\max\{C_1', C_2', C_3', C_4'\}$.
By the setting of $\eta_t = \eta$, we have $1/\eta_t - 1/\eta_{t-1}=0$, with $\beta_{gy, t+1} = \frac{8\tau}{\lambda}\frac{\tau_tC\gamma}{\eta}\eta^2$ (where $\beta_{gy, t+1}< 1$ for small enough  constant $\gamma$ and large enough constant $c_0$,  where $C$ is given below).  
Set $\beta_{fx, t+1} =  3C_1 \gamma \eta^2$, 
$\beta_{fy, t+1} =  3C_4 \gamma \eta^2$, 
$\beta_{gxy, t+1} =  3C_2 \gamma \eta^2$, 
$\beta_{gyy, t+1} =  3C_3 \gamma \eta^2$, 
$\beta_{t+1} = 2\gamma C'_{max} \eta^2$,  
$\frac{C\tau_t\tau\lambda}{4} \geq C_0\eta$  and $\frac{C\gamma}{4\tau_t}\geq \frac{6L(\p)}{\eta}$. 
With $\eta\leq \min \{\frac{1}{\sqrt{3C_1 \gamma}},\frac{1}{\sqrt{3C_2 \gamma}},\frac{1}{\sqrt{3C_3 \gamma}}, \frac{1}{\sqrt{3C_4 \gamma}}, \frac{1}{\sqrt{\gamma C_{max}'}}\}$, all $\beta_{*, t+1} < 1$.

Adding all recursions for $\delta_{y, t} = \|\y_{t+1} - \y*(\x_t)\|^2, \delta_{gy, t}=\|\w_{t+1} - \nabla_y g(\x_t, \y_t)\|^2, \delta_{fx,t}=\|\u_{t+1} - \nabla_x f(\x_t, \y_t)\|^2$, $\delta_{fy,t}=\|\v_{t+1} - \nabla_y f(\x_t; \y_t)\|^2$, $\delta_{gxy, t}= \|V_{t+1} - \nabla_{xy}^2 g(\x_t, \y_t)\|_F^2$, $\delta_{gyy,t} = \|H_{t+1} - \nabla_{yy}^2 g(\x_t, \y_t)\|^2$ and $\delta_{d,t} = \|\d_{t+1} - \frac{1}{m}\sum_{i=1}^m\z_{i,t+1}\|^2$, we have
\begin{align*} 
&\E\bigg[\sum_{t=0}^T\frac{C\gamma}{m}(\delta_{y,t+1} - \delta_{y,t})\bigg] + \E\bigg[\sum_{t=0}^T(\frac{\delta_{gy, t+1}}{m\eta} - \frac{\delta_{gy, t}}{m\eta_{t-1}})\bigg] + \E\bigg[\sum_{t=0}^T(\frac{\delta_{fx, t+1}}{m\eta} -  \frac{\delta_{fx, t}}{m\eta_{t-1}})\bigg] \\
&+  \E\bigg[\sum_{t=0}^T(\frac{\delta_{d, t+1}}{C_{\max}'\eta} - \frac{\delta_{d, t}}{C'_{\max}\eta_{t-1}})\bigg] +  \E\bigg[\sum_{t=0}^T(\frac{\delta_{fy, t+1}}{m\eta} - \frac{\delta_{fy, t}}{m\eta_{t-1}})\bigg] \\
& + \E\bigg[\sum_{t=0}^T(\frac{\delta_{gxy, t+1}}{m\eta} - \frac{\delta_{gxy, t}}{m\eta_{t-1}})\bigg]  +  \E\bigg[\sum_{t=0}^T(\frac{\delta_{gyy, t+1}}{m\eta} - \frac{\delta_{gyy, t}}{m\eta_{t-1}})\bigg] \\
&\leq \sum_{t=0}^T\frac{2(\beta_{gy, t+1}^2+ \beta_{fx,t+1}^2 + \beta_{fy,t+1}^2 + \beta_{gxy, t+1}^2 + \beta_{gyy,t+1}^2 )C(\p)}{m\eta}  + \frac{2\beta_{t+1}^2\sigma^2}{C'_{\max}\eta}+  \sum_{t=0}^T(- \frac{C\gamma\tau_t\tau\lambda}{4m}) \E[\delta_{y,t}]\\
&- \sum_{t=0}^T 3C_1\gamma\eta \frac{\E[\delta_{fx,t}]}{m} -  \sum_{t=0}^T 3C_4 \gamma \eta \frac{\delta_{fy,t}}{m} 
- \sum_{t=0}^T 3C_2\gamma \eta \frac{\E[\delta_{gxy,t}]}{m} -\sum_{t=0}^T 3C_3 \gamma\eta \frac{\E[\delta_{gyy,t}]}{m}  
- \sum_{t=0}^T 2\gamma\eta \E[\delta_{d,t}]\\ 
& + \sum_{t=0}^T\frac{8C\gamma^3\eta^2L_y^2}{\tau_t\tau\lambda m}\E[\|\d_{t+1}\|^2]  - \sum_{t=0}^T \frac{2C\gamma\tau}{\tau_t m}(1+ \frac{\tau_t\tau \lambda}{4})(\frac{1}{2\tau} - \frac{3L_g}{4} )\E[\|\y_t- \y_{t+1}\|^2]\\  
&+  \sum_{t=0}^T\frac{2L(\p)}{m\eta}\E(\|\x_{t+1} - \x_t\|^2 + \|\y_{t+1} - \y_t\|^2)\\
& + \frac{1}{8m}\sum_{t=0}^{T}\left[\frac{1}{\eta}\E[\|\u_{t+2} - \u_{t+1}\|^2] + \frac{1}{\eta}\E[\|V_{t+2} - V_{t+1}\|_F^2]\right] \\
&+  \frac{1}{8m}\sum_{t=0}^{T}\left[\frac{1}{\eta}[\E\|H_{t+2} - H_{t+1}\|_F^2]   +  \frac{1}{\eta}[\E\|\v_{t+2} - \v_{t+1}\|^2]\right]
\end{align*} 

For the last four terms, we can bound them by
\begin{equation}
\begin{split}
&\frac{1}{m}\sum\limits_{t=0}^{T}(\frac{\beta^2_{fx,t+1}}{2\eta}\delta_{fx,t} + \frac{\beta^2_{gxy,t+1}}{2\eta}\delta_{gxy,t} + \frac{\beta^2_{gyy,t+1}}{2\eta}\delta_{gyy,t} + \frac{\beta^2_{fy,t+1}}{2\eta}\delta_{fy,t} ) \\
&+\frac{1}{m}\sum\limits_{t=0}^{T}(\frac{\beta^2_{fx,t+1}C_{fx}(\p)}{4\eta} + \frac{\beta^2_{gxy,t+1} C_{gxy}(\p)}{4\eta} + \frac{\beta^2_{gyy,t+1}C_{gyy}(\p)}{4\eta} + \frac{\beta^2_{fy,t+1} C_{fy}(\p)}{4\eta} ) \\
& +\frac{1}{m}\sum\limits_{t=0}^{T}\frac{(L_{fx}(\p)+L_{fx}+L_{fy}(\p)+L_{fy}+ L_{gxy}(\p) + L_{gxy} + L_{gyy}(\p)+L_{gyy})} {2\eta}\E(\|\x_{t+1}-\x_t\|^2\\
& + \|\y_{t+1}-\y_t\|^2).
\end{split}
\end{equation}

With $\gamma (\frac{8C \gamma L_y^2}{\tau\lambda m} + 3\frac{L(\p)}{m})\leq 1/4$ (by $\gamma\leq \min\{\frac{\sqrt{\tau\lambda m}}{8\sqrt{C}L_y},  \frac{m}{24L(\p)}\}$ ) and $C\geq\max\{\frac{6L(\p)}{\gamma \tau}, 12C_0 \tau \lambda\}$, we have  
\begin{align*} 
&\E\bigg[\sum_{t=0}^T\frac{C\gamma}{m}(\delta_{y,t+1} - \delta_{y,t})\bigg] + \E\bigg[\sum_{t=0}^T(\frac{\delta_{gy, t+1}}{m\eta} - \frac{\delta_{gy, t}}{m\eta})\bigg] + \E\bigg[\sum_{t=0}^T(\frac{\delta_{fx, t+1}}{m\eta} - \frac{\delta_{fx, t}}{m\eta})\bigg] \\
&+  \E\bigg[\sum_{t=0}^T(\frac{\delta_{d, t+1}}{C'_{\max}\eta} - \frac{\delta_{d, t}}{C'_{\max}\eta})\bigg]
+  \E\bigg[\sum_{t=0}^T(\frac{\delta_{fy, t+1}}{m\eta} - \frac{\delta_{fy, t}}{m\eta})\bigg]\\
& + \E\bigg[\sum_{t=0}^T(\frac{\delta_{gxy, t+1}}{m\eta} - \frac{\delta_{gxy, t}}{m\eta})\bigg]  + \E\bigg[\sum_{t=0}^T(\frac{\delta_{gyy, t+1}}{m\eta} - \frac{\delta_{gyy, t}}{m\eta})\bigg] \\ 
&\leq \sum_{t=0}^T\frac{3(\beta_{gy, t+1}^2+ \beta_{fx,t+1}^2 + \beta_{fy,t+1}^2 + \beta_{gxy, t+1}^2 + \beta_{gyy,t+1}^2)C(\p)}{m\eta} + \frac{2\beta_{t+1}^2\sigma^2}{C'_{\max}\eta} 
+ \frac{1}{4}\gamma \eta \E\|\d_{t+1}\|^2\\
& -  \sum_{t=0}^T3C_0\gamma \eta\frac{\E[\delta_{y,t}]}{m}-  \sum_{t=0}^T3C_1\gamma\eta\frac{\E[\delta_{fx,t}]}{m} -  \sum_{t=0}^T3C_4\gamma\eta \frac{\E[\delta_{fy,t}]}{m}-  \sum_{t=0}^T3C_2\gamma\eta \frac{\E[\delta_{gxy,t}]}{m} \\
& -   \sum_{t=0}^T3C_3\gamma\eta  \frac{\E[\delta_{gyy,t}]}{m} - \sum_{t=0}^T2\gamma\eta \E[\delta_{d,t}].  
\end{align*}
Adding the above inequality with that in  Lemma~\ref{lem:1} (noting that notation $\d$ in this subsection is corresponding to $\z$ in Lemma~\ref{lem:1}), we have 
\begin{equation}
\begin{split}
&\E\bigg[\sum_{t=0}^T \frac{\gamma\eta}{2}\|\nabla F(\x_t)\|^2 \bigg]+\E\bigg[\sum_{t=0}^T\frac{C\gamma}{m}(\delta_{y,t+1} - \delta_{y,t})\bigg] + \E\bigg[\sum_{t=0}^T\frac{\delta_{gy, t+1}}{m\eta} - \frac{\delta_{gy, t}}{m\eta}\bigg]\\
&+ \E\bigg[\sum_{t=0}^T\frac{\delta_{fx, t+1}}{m\eta} - \frac{\delta_{fx, t}}{m\eta}\bigg] +  \E\bigg[\sum_{t=0}^T\frac{\delta_{d, t+1}}{C'_{\max}\eta} - \frac{\delta_{d, t}}{C'_{\max}\eta}\bigg]\\
& +  \E\bigg[\sum_{t=0}^T\frac{\delta_{fy, t+1}}{m\eta} - \frac{\delta_{fy, t}}{m\eta}\bigg]+ \E\bigg[\sum_{t=0}^T\frac{\delta_{gxy, t+1}}{m\eta} - \frac{\delta_{gxy, t}}{m\eta}\bigg]  + \E\bigg[\sum_{t=0}^T\frac{\delta_{gyy, t+1}}{m\eta} - \frac{\delta_{gyy, t}}{m\eta}\bigg] \\
&  \leq \E\bigg[F(\x_0)  - F(\x_*) \bigg] + \gamma^2\frac{C(\p)}{m}O(1)\sum_{t=0}^T\eta^3 -  \sum_{t=0}^TC_0\gamma \eta\frac{\E[\delta_{y,t}]}{m}-  \sum_{t=0}^TC_1\gamma\eta\frac{\E[\delta_{fx,t}]}{m} \\
&~~~ -  \sum_{t=0}^TC_4\gamma\eta \frac{\E[\delta_{fy,t}]}{m}-  \sum_{t=0}^TC_2\gamma\eta \frac{\E[\delta_{gxy,t}]}{m} -   \sum_{t=0}^T C_3\gamma\eta  \frac{\E[\delta_{gyy,t}]}{m} - \sum_{t=0}^T \gamma\eta \E[\delta_{d,t}].  
\end{split}
\end{equation}

As a result, 
\begin{equation}
\label{equ:one_stage}
\begin{split}
&\E\bigg[\sum_{t=0}^T \frac{\gamma\eta}{2}\|\nabla F(\x_t)\|^2\bigg] +  \sum_{t=0}^TC_0\gamma \eta\frac{\E[\delta_{y,t}]}{m}+ \sum_{t=0}^TC_1\gamma\eta\frac{\E[\delta_{fx,t}]}{m} +  \sum_{t=0}^TC_4\gamma\eta \frac{\E[\delta_{fy,t}]}{m}+  \sum_{t=0}^TC_2\gamma\eta \frac{\E[\delta_{gxy,t}]}{m} \\
&+ \sum_{t=0}^T C_3\gamma\eta  \frac{\E[\delta_{gyy,t}]}{m} + \sum_{t=0}^T \gamma\eta \E[\delta_{d,t}] \\
& \leq \Delta + \frac{C\gamma \E[\delta_{y, 0}]}{m} + \frac{\E[\delta_{gy,0} +\delta_{fx,0} +  \delta_{fy,0} + \delta_{gxy,0} + \delta_{gyy,0}]}{\eta m}   +  \frac{\E[\delta_{d,0}]}{C'_{\max}\eta} 
+ \gamma^2\frac{C(\p)}{m}C_5\eta^3(T+1), 
\end{split} 
\end{equation}
where $\Delta > \E[F(\x_0) - F(\x_*)$ and 
$C_5 = 3(9C_1^2 + 9C_4^2 + 9C_2^2+9C_3^2) +  \frac{8C'_{max} \sigma^2 m}{C(\p)} \in O(1)$.  
Dividing both sides by $\gamma \eta (T+1)/2$, we have
\begin{align*}
\E\bigg[\frac{1}{T+1}\sum_{t=0}^T\|\nabla F(\x_t)\|^2\bigg]\leq& \frac{2\Delta}{\gamma \eta T} + \frac{2\E(\delta_{gy,0} +\delta_{fx,0} +  \delta_{fy,0} + \delta_{gxy,0} + \delta_{gyy,0})}{\eta  m \gamma \eta T}   + \frac{\E[\delta_{d,0}]}{C'_{\max}\eta \gamma\eta T} \\
& + \gamma\frac{C(\p)}{m}C_5 \eta^2 + \frac{2C \E[\delta_{y, 0}]}{m\eta T}. 
\end{align*}

Note that $\gamma/m = O(1/L(\p)) = O(1/C(\p))$.
By choosing $\eta = \frac{\epsilon}{2}\sqrt{\frac{m}{\gamma C(\p) C_5}}=O(\epsilon)$, $T = O(C(\p)/(m\epsilon^3))$, $\u_0, \v_0, \w_0, V_0, H_0, d_0$ such that the maximum of $\E[\delta_{gy,0}], \E[\delta_{fx,0}],  \E[\delta_{fy,0}],$ $\E[\delta_{gxy,0}], \E[\delta_{gyy,0}] \leq \eta m = O(m\epsilon)$, which can be achieved by using a large mini-batch at the beginning $T_0 = 1/(\eta m) = O(\frac{1}{m\epsilon})$, and $\d_0 = 0$, which can be achieved by processing all lower problems at the beginning, $\delta_{y,0}\leq m/\gamma = O(L(\p))$, which can be achieved by computing the finding a good initial solution $\y_0$ with accuracy $O(L(\p))$ with a complexity of $O(1/(L(\p)))$. In the uniform sampling case, $C(\p)/m = O(m), L(\p)/ m = O(m)$. The worse case iteration complexity is $T + T_0=O(m/\epsilon^3)$. 

\end{proof}

\section{Proof of Theorem \ref{thm:main}}
\begin{proof}
Applying (\ref{equ:one_stage}) in Appendix \ref{sec:thm2}, and defining a constant 
$C_6 = \min\{2, C_0, C_1, C_2, C_3, C_4\}$, we obtain that for any of the $K$ epochs, 
\begin{equation}
\begin{split}
&\E\bigg[\sum_{t=0}^T C_6 \gamma\eta \|\nabla F(\x_t)\|^2 +  \sum_{t=0}^TC_6\gamma \eta\frac{\E[\delta_{y,t}]}{m}+  \sum_{t=0}^TC_6\gamma\eta\frac{\E[\delta_{fx,t}]}{m} +  \sum_{t=0}^TC_6\gamma\eta \frac{\E[\delta_{fy,t}]}{m}\\
&~~~ +  \sum_{t=0}^TC_6\gamma\eta \frac{\E[\delta_{gxy,t}]}{m} 
+ \sum_{t=0}^T C_6\gamma\eta  \frac{\E[\delta_{gyy,t}]}{m} + \sum_{t=0}^T C_6 \gamma\eta \E[\delta_{d,t}] \bigg]\\ 
& \leq \Delta + \frac{C\gamma \E[\delta_{y, 0}]}{m} + \frac{\E[\delta_{gy,0} +\delta_{fx,0} +  \delta_{fy,0} + \delta_{gxy,0} + \delta_{gyy,0}]}{\eta m}   +  \frac{\E[\delta_{d,0}]}{C'_{\max}\eta} \\
&~~~ + \gamma^2\frac{C(\p)}{m}C_5\eta^3(T+1). 
\end{split} 
\label{equ:one_stage_use}
\end{equation}

From Theorem \ref{thm:2}, we know that it is required that $\eta_k\leq \min\{\frac{1}{\sqrt{3(C_1+C_2+C_3+C_4+C'_{max})\gamma}}, \frac{1}{2L_F \gamma}\}:=\eta_0$. 
Without loss of generality, let us assume that $\epsilon_0=\Delta > \frac{5\eta_0^2\gamma C(\p) C_5}{C_6 m \mu}$, i.e.,  $\sqrt{\frac{C_6m\mu\epsilon_0}{5\gamma C(\p) C_5}} > \eta_0$. The case that $\Delta \leq  \frac{5\eta_0^2\gamma C(\p) C_5}{C_6 m \mu}$can be simply covered by our proof. Then denote $\epsilon_1 =  \frac{5\eta_0^2\gamma C(\p) C_5}{C_6 m \mu}$ and $\epsilon_k=\epsilon_1/2^{k-1}$.  

In the first epoch ($k=1$), we have initialization such that $F(\x_0)-F(\x_*)\leq \Delta$. In the following, we let the last subscript denote the epoch index. Setting $\eta_1 = \eta_0$ and $T_1=\max\{\frac{5\Delta}{\mu C_6 \eta_1 \gamma}, \frac{5C\gamma\E[\delta_{y,0}]}{\mu C_6 m \eta_1 \gamma}, \frac{5\E[\delta_{gy,0}+\delta_{fx,0}+\delta_{fy,0}+\delta_{gxy,0}+\delta_{gyy,0}]}{\mu C_6},\frac{5\E[\delta_{d,0}]}{\mu C_6 C_{max}' \eta_1^2 \gamma}\}$, we bound the error of the first stage's output as follows,
\begin{equation} 
\begin{split}    
&\E\bigg[\|\nabla F(\x_1)\|^2\bigg] +  \frac{\E[\delta_{y,1}]}{m}+   \frac{\E[\delta_{fx,1}]}{m} +   \frac{\E[\delta_{fy,1}]}{m}+  \frac{\E[\delta_{gxy,1}]}{m} +   \frac{\E[\delta_{gyy,1}]}{m} + \E[\delta_{d,1}] \\ 
& \leq \frac{\Delta}{C_6 \eta_1\gamma T_1} + \frac{C\gamma \E[\delta_{y, 0}]}{C_6 m \eta_1\gamma T_1} + \frac{\E[\delta_{gy,0} +\delta_{fx,0} +  \delta_{fy,0} + \delta_{gxy,0} + \delta_{gyy,0}]}{C_6 m \eta_1^2\gamma T_1}  \\  
&~~~ +  \frac{\E[\delta_{d,0}]}{C_6 C'_{\max}\eta_1^2 \gamma T_1}  
+ \gamma\frac{C(\p)}{C_6 m}C_5\eta_1^2\\  
&\leq \frac{5\eta_0^2 \gamma C(\p)C_5}{C_6 m} = \mu \epsilon_1,
\end{split} 
\end{equation} 
where the first inequality uses  (\ref{equ:one_stage_use}) and the fact that the output of each epoch is randomly sampled from all iterations, and the last line uses the choice of $\eta_1, T_1, \epsilon_1$.  
It follows that
\begin{equation}
\begin{split}
\E[F(\x_1) - F(\x_*)] \leq \frac{1}{2\mu} \E[\|\nabla F(\x_1)\|^2] \leq \frac{\epsilon_1}{2}. 
\end{split}    
\end{equation}

Starting from the second stage, we  will prove by induction. Suppose we are at $k$-th stage. Assuming that the output of $(k-1)$-th stage satisfies that  $\E[F(\x_{k-1}) - F(\x_*)] \leq \epsilon_{k-1}$ and $\E[\frac{\delta_{y,k-1}}{m} +  \frac{\E[\delta_{gy,k-1}+\delta_{fx,k-1}+\delta_{fy,k-1}+\delta_{gxy,k-1} + \delta_{gyy,k-1}}{m} +\delta_{d,k-1}] \leq \mu\epsilon_{k-1}$. 
We have 
\begin{equation}
\begin{split}
&\E[\|\nabla F(\x_k)\|^2] +  \frac{\E[\delta_{y,k}]}{m}+   \frac{\E[\delta_{fx,k}]}{m} +   \frac{\E[\delta_{fy,k}]}{m}+  \frac{\E[\delta_{gxy,k}]}{m} +   \frac{\E[\delta_{gyy,k}]}{m} + \E[\delta_{d,1}] \\ 
& \leq \frac{\E[F(\x_{k-1})-F(\x_*)]}{C_6 \eta_k\gamma T_k} + \frac{C\gamma \E[\delta_{y, k-1}]}{C_6 m \eta_k\gamma T_k} +  \frac{\E[\delta_{d,k-1}]}{C_6 C'_{\max}\eta_1^2\gamma T_k} \\
&+ \frac{\E[\delta_{gy,k-1} +\delta_{fx,k-1} +  \delta_{fy,k-1} + \delta_{gxy,k-1} + \delta_{gyy,k-1}]}{C_6 m \eta_k^2\gamma T_k} 
+ \gamma\frac{C(\p)}{C_6 m}C_5\eta_k^2\\
&\leq \frac{\epsilon_{k-1}}{C_6\eta_k\gamma T_k}
+ \frac{C\gamma \mu\epsilon_{k-1}}{C_6 m \eta_k\gamma T_k} + \frac{\mu\epsilon_{k-1}}{C_6 m \eta_k^2\gamma T_k}  +  \frac{\mu\epsilon_{k-1}}{C_6 C'_{\max}\eta_k^2\gamma T_k} 
+ \gamma\frac{C(\p)}{C_6 m}C_5\eta_k^2\\ 
&\leq \mu \epsilon_k, 
\end{split}
\end{equation}
where the last inequality holds by the setting $\eta_k = \sqrt{\frac{C_6 m \mu\epsilon_k}{5\gamma C(\p) C_5}}$, $T_k=\max\{\frac{10}{\mu C_6\eta_k\gamma}, \frac{10C\gamma}{C_6m\eta_k\gamma}, \frac{10}{C_6 m \eta_k^2 \gamma}, \frac{10}{C_6 C_{max}'\eta_k^2 \gamma}\}$. 

It follows that
\begin{equation}
\begin{split}
\E[F(\x_k) - F(\x_*)] \leq \frac{1}{2\mu} \E[\|\nabla F(\x_k)\|^2] \leq \frac{\epsilon_k}{2}. 
\end{split}    
\end{equation}

Thus, after $K = 1+\log_2(\epsilon_1/\epsilon) \leq \log_2(\epsilon_0/\epsilon)$ stages, $\E[F(\x_k) - F(\x_*)] \leq \epsilon$.

\end{proof}

\section{Proof of Lemmas}

\subsection{Useful Lemmas}
\begin{lemma}(\citep{99401}[Lemma 2.2])Under Assumptions~\ref{ass:1}, ~\ref{ass:2} or Assumptions~\ref{ass:3}, ~\ref{ass:4}, we have
\begin{align*}
&\|\nabla F(\x) - \nabla F(\x')\|\leq L_F\|\x- \x'\|\\
&\|\y^*(\x) - \y^*(\x')\|\leq L_y \|\x - \x'\|
\end{align*}
where $L_y$ and $L_F$ are appropriate constants. 
\end{lemma}

\subsection{Proof of Lemma~\ref{lem:1}}
\begin{proof}
Due the smoothness of $F$, we can prove that under $\gamma\eta_t L_F\leq 1/2$
\begin{align*}
&F(\x_{t+1}) \leq F(\x_t) + \nabla F(\x_t)^{\top} (\x_{t+1} - \x_t) + \frac{L_F}{2}\|\x_{t+1} - \x_t\|^2\\
&= F(\x_t) - \gamma\eta_t\nabla F(\x_t)^{\top} \z_{t+1} + \frac{L_F\gamma^2\eta_t^2}{2}\|\z_{t+1}\|^2\\
&  = F(\x_t) +   \frac{ \gamma\eta_t}{2}\|\nabla F(\x_t) - \z_{t+1}\|^2- \frac{\gamma\eta_t}{2}\|\nabla F(\x_t)\|^2 + (\frac{L_F\gamma^2\eta_t^2}{2} - \frac{\gamma\eta_t}{2})\|\z_{t+1}\|^2\\
&  \leq F(\x_t) +   \frac{ \gamma\eta_t}{2}\|\nabla F(\x_t) - \z_{t+1}\|^2- \frac{\gamma\eta_t}{2}\|\nabla F(\x_t)\|^2  - \frac{\gamma\eta_t}{4}\|\z_{t+1}\|^2
\end{align*}
\end{proof}

\subsection{Proof of Lemma~\ref{lem:2}}
\begin{proof}
First, we have 
\begin{align*}
&\|\nabla F(\x_t, \y_t) - \nabla F(\x_t)\|^2 \\
&\leq \|\nabla_x f(\x_t, \y_t) - \nabla^2_{xy}g(\x_t, \y_t)[\nabla_{yy}^2g(\x_t, \y_t)]^{-1}\nabla_y f(\x_t, \y_t) \\
&- \nabla_x f(\x_t, \y^*(\x_t)) - \nabla^2_{xy}g(\x_t, \y^*(\x_t))[\nabla_{yy}^2g(\x_t, \y^*(\x_t))]^{-1}\nabla_y f(\x_t, \y^*(\x_t)) \|^2\\
& \leq 2\|\nabla_x f(\x_t, \y_t) -   \nabla_x f(\x_t, \y^*(\x_t))\|^2 \\
&+ 2\| \nabla^2_{xy}g(\x_t, \y_t)[\nabla_{yy}^2g(\x_t, \y_t)]^{-1}\nabla_y f(\x_t, \y_t) -  \nabla^2_{xy}g(\x_t, \y^*(\x_t))[\nabla_{yy}^2g(\x_t, \y^*(\x_t))]^{-1}\nabla_y f(\x_t, \y^*(\x_t))\|^2\\
&\overset{(a)}{\leq}2\|\nabla_x f(\x_t, \y_t) -   \nabla_x f(\x_t, \y^*(\x_t))\|^2\\
&  + 6\| \nabla^2_{xy}g(\x_t, \y_t)[\nabla_{yy}^2g(\x_t, \y_t)]^{-1}\nabla_y f(\x_t, \y_t) -  \nabla^2_{xy}g(\x_t, \y^*(\x_t))[\nabla_{yy}^2g(\x_t, \y_t)]^{-1}\nabla_y f(\x_t, \y_t)\|^2\\
& + 6\| \nabla^2_{xy}g(\x_t, \y^*(\x_t))[\nabla_{yy}^2g(\x_t, \y_t)]^{-1}\nabla_y f(\x_t, \y_t)-  \nabla^2_{xy}g(\x_t, \y^*(\x_t))[\nabla_{yy}^2g(\x_t, \y^*(\x_t))]^{-1}\nabla_y f(\x_t, \y_t)\|^2\\
& + 6\| \nabla^2_{xy}g(\x_t, \y^*(\x_t))[\nabla_{yy}^2g(\x_t, \y^*(\x_t))]^{-1}\nabla_y f(\x_t, \y_t)
-  \nabla^2_{xy}g(\x_t, \y^*(\x_t))[\nabla_{yy}^2g(\x_t,  \y^*(\x_t))]^{-1}\nabla_y f(\x_t, \y^*(\x_t))\|^2\\
&\overset{(a)}{\leq}2L_{fx}^2\|\y_t - \y^*(\x_t)\|^2 + \frac{6C_{fy}^2L_{gxy}^2}{\lambda^2}\|\y_t - \y^*(\x_t)\|^2 + \frac{6C_{fy}^2C^2_{gxy}L^2_{gyy}}{\lambda^4}\|\y_t - \y^*(\x_t)\|^2 + \frac{6L_{fy}^2C^2_{gxy}}{\lambda^2}\|\y_t - \y^*(\x_t)\|^2\\
& = \underbrace{ (2L_{fx}^2 +  \frac{6C_{fy}^2L_{gxy}^2}{\lambda^2} + \frac{6C_{fy}^2C^2_{gxy}L^2_{gyy}}{\lambda^4}+ \frac{6L_{fy}^2C^2_{gxy}}{\lambda^2})}\limits_{C_0}\|\y_t - \y^*(\x_t)\|^2, 
\end{align*}
where (a) uses the Lipschitz continuity of $\nabla_x f(\x, \cdot)$, $\nabla_y f(\x, \cdot)$, $\nabla_{xy}^2 g(\x, \cdot)$, $\nabla_{yy}^2 g(\x, \cdot)$, and upper bound of $\|\nabla_y f(\x, \y)\|$, $\|\nabla_{xy}^2 g(\x, \y)\|$, and $\|\nabla_{yy}^2 g(\x, \y)\|\succeq \lambda I $.  Next, 
\begin{align*}
&\|\u_{t+1} -V_{t+1}[H_{t+1}]^{-1}\v_{t+1}  - \nabla F(\x_t, \y_t)\|^2\\
& = \| \u_{t+1} -V_{t+1}[H_{t+1}]^{-1}\v_{t+1}-  \nabla_x f(\x_t, \y_t) - \nabla^2_{xy}g(\x_t, \y_t)[\nabla_{yy}^2g(\x_t, \y_t)]^{-1}\nabla_y f(\x_t, \y_t)\|^2\\
&\leq 2\|\u_{t+1} -  \nabla_x f(\x_t, \y_t) \|^2 + 2\|V_{t+1}[H_{t+1}]^{-1}\v_{t+1} - \nabla^2_{xy}g(\x_t, \y_t)[\nabla_{yy}^2g(\x_t, \y_t)]^{-1}\nabla_y f(\x_t, \y_t)\|^2\\
&\leq 2\|\u_{t+1} -  \nabla_x f(\x_t, \y_t) \|^2 \\
&+ 6\|V_{t+1}[H_{t+1}]^{-1}\v_{t+1} - \nabla^2_{xy}g(\x_t, \y_t)[H_{t+1}]^{-1}\v_{t+1}\|^2\\
&+ 6\| \nabla^2_{xy}g(\x_t, \y_t)[H_{t+1}]^{-1}\v_{t+1} - \nabla^2_{xy}g(\x_t, \y_t)[\nabla_{yy}^2g(\x_t, \y_t)]^{-1}\v_{t+1}\|^2\\
&+ 6\| \nabla^2_{xy}g(\x_t, \y_t)[\nabla_{yy}^2g(\x_t, \y_t)]^{-1}\v_{t+1} -  \nabla^2_{xy}g(\x_t, \y_t)[\nabla_{yy}^2g(\x_t, \y_t)]^{-1}\nabla_y f(\x_t, \y_{t})\|^2\\
&\leq 2\|\u_{t+1} -  \nabla_x f(\x_t, \y_t) \|^2 
+ \frac{6C_{fy}^2}{\lambda^2}\|V_{t+1} - \nabla^2_{xy}g(\x_t, \y_t)\|^2  \\
&~~~
+ \frac{6C_{fy}^2C_{gxy}^2}{\lambda^4}\|H_{t+1} - \nabla_{yy}^2g(\x_t, \y_t)\|^2  +  \frac{6C_{gxy}^2}{\lambda^2}\|\v_{t+1} -\nabla_y f(\x_t, \y_t) \|^2 \\ 
&\leq C_1\|\u_{t+1} -  \nabla_x f(\x_t, \y_t) \|^2 
+ C_2\|V_{t+1} - \nabla^2_{xy}g(\x_t, \y_t)\|^2 \\ 
&~~~ +C_3\|H_{t+1} - \nabla_{yy}^2g(\x_t, \y_t)\|^2 +C_4\|\v_{t+1} -\nabla_y f(\x_t, \y_t) \|^2,
\end{align*}
where $C_1 = 2$, $C_2 = \frac{6C_{fy}^2}{\lambda^2}$, $C_3 = \frac{6C_{fy}^2C_{gxy}^2}{\lambda^4}$ and $C_4 = \frac{6C_{gxy}^2}{\lambda^2}$. 
As a result, 
\begin{align*}
&\|\u_{t+1} -V_{t+1}[H_{t+1}]^{-1}\v_{t+1}  - \nabla F(\x_t)\|^2\leq 2\|\u_{t+1} -V_{t+1}[H_{t+1}]^{-1}\v_{t+1}  - \nabla F(\x_t, \y_t)\|^2 \\
&+ 2\| \nabla F(\x_t, \y_t) - \nabla F(\x_t)\|^2\leq 2\|\u_{t+1} -V_{t+1}[H_{t+1}]^{-1}\v_{t+1}  - \nabla F(\x_t, \y_t)\|^2  + 2C_1\|\y^*(\x_t) - \y_t\|^2\\
&\leq 2C_0\|\y_t - \y^*(\x_t)\|^2 + 2C_1\|\u_{t+1} -  \nabla_x f(\x_t, \y_t) \|^2  + 2C_2\|V_{t+1} - \nabla^2_{xy}g(\x_t, \y_t)\|^2 \\
&~~~ +2C_3\|H_{t+1} - \nabla_{yy}^2g(\x_t, \y_t)\|^2 +2C_4\|\v_{t+1} -\nabla_y f(\x_t, \y_t) \|^2. 
\end{align*} 
\end{proof}

\subsection{Proof of Lemma~\ref{lem:3}}
\begin{proof}
Define $\widetilde\y_{t+1}=\y_t -\tau \w_{t+1}$. $\y_{t+1}  = \y_t + \tau_t (\widetilde\y_{t+1} - \y_t)$
\begin{equation*}
\begin{aligned}
&\|\y_{t+1}- \y^*(\x_t) \|^2\leq \|\y_t - \tau_t(\y_t -  \widetilde\y_{t+1}) - \y^*(\x_t)\|^2 \\
& = \|\y_t - \y^*(\x_t)\|^2 + \tau_t^2 \|\y_t -  \widetilde\y_{t+1}\|^2 +  2\tau_t(\widetilde\y_{t+1} - \y_t)^{\top}(\y_t -  \y^*(\x_t)). 
\end{aligned}
\end{equation*}
As a result, 
\begin{align}\label{eqn:yy}
(\widetilde\y_{t+1} - \y_t)^{\top}(\y_t -  \y^*(\x_t))\!\geq\! \frac{1}{2\tau_t}(\|\y_{t+1}- \y^*(\x_t) \|^2 \!-\!  \|\y_t - \y^*(\x_t)\|^2 \!-\! \tau_t^2 \|\y_t -  \widetilde\y_{t+1}\|^2).
\end{align} 

Due to smoothness of $g$, we have
\begin{align*}
g(\x_t, \widetilde\y_{t+1})\leq g(\x_t, \y_t)+ \nabla_y g(\x_t, \y_t)^{\top}(\widetilde\y_{t+1}- \y_t) + \frac{L_g}{2}\|\widetilde\y_{t+1} - \y_t\|^2
\end{align*}
Hence 
\begin{align*}
g(\x_t, \widetilde\y_{t+1}) -  \nabla_y g(\x_t, \y_t)^{\top}(\widetilde\y_{t+1}- \y_t) -  \frac{L_g}{2}\|\widetilde\y_{t+1} - \y_t\|^2
\leq g(\x_t, \y_t)\end{align*}
Due to strong convexity of $g$, we have
\begin{align*}
&g(\x_t, \y)\geq g(\x_t, \y_t) + \nabla_y g(\x_t, \y_t)^{\top}(\y - \y_t) + \frac{\lambda}{2}\|\y - \y_t\|^2\\
& = g(\x_t, \y_t) + \nabla_y g(\x_t, \y_t)^{\top}(\y -\widetilde\y_{t+1}) + \nabla_y g(\x_t, \y_t)^{\top}(\widetilde\y_{t+1} - \y_t) + \frac{\lambda}{2}\|\y - \y_t\|^2\\
&= g(\x_t, \y_t) + \w_{t+1}^{\top}(\y - \widetilde\y_{t+1})+ (\nabla_y g(\x_t, \y_t) -\w_{t+1})^{\top}(\y - \widetilde\y_{t+1})\\
&  + \nabla_y g(\x_t, \y_t)^{\top}(\widetilde\y_{t+1} - \y_t)  + \frac{\lambda}{2}\|\y - \y_t\|^2
\end{align*}
Combining the above inequalities we have
\begin{equation}
\begin{split} 
&g(\x_t, \y)\geq g(\x_t, \y_t) + \w_{t+1}^{\top}(\y -  \widetilde\y_{t+1})+ (\nabla_y g(\x_t, \y_t) -\w_{t+1})^{\top}(\y -  \widetilde\y_{t+1}) \\
&+  \nabla_y g(\x_t, \y_t)^{\top}(\widetilde\y_{t+1} - \y_t) + \frac{\lambda}{2}\|\y - \y_t\|^2\\
&= g(\x_t, \widetilde\y_{t+1})  + \w_{t+1}^{\top}(\y - \widetilde\y_{t+1}) + (\nabla_y g(\x_t, \y_t) -\w_{t+1})^{\top}(\y -  \widetilde\y_{t+1}) \\
&~~~ + \frac{\lambda}{2}\|\y - \y_t\|^2 -    \frac{L_g}{2}\| \widetilde\y_{t+1} - \y_t\|^2. 
\label{equ:local:proj_pause} 
\end{split}
\end{equation}

Note that $\tilde{\y}_{t+1} = \Pi_{\mathcal{Y}} (\y_t - \tau \w_{t+1}) = \arg\min\limits_{\y\in\mathcal{Y}} \frac{1}{2}\|\y-\y_t + \tau\w_{t+1}\|^2$. Since $\mathcal{Y}$ is a convex set and the function $\frac{1}{2}\|\y - \y_t + \tau \w_{t+1}\|^2$ is convex in $\y$, according to the first order optimality condition of convex function, we have
\begin{equation}
\begin{split}
(\tilde{\y}_{t+1} - \y_t + \tau\w_{t+1})^\top (\y-\widetilde{\y}_{t+1}) \geq 0, \forall \y\in \mathcal{Y}.
\end{split} 
\end{equation} 

Then we obtain
\begin{equation}
\begin{split}
\w_{t+1}^\top (\y-\widetilde{\y}_{t+1}) 
&\geq \frac{1}{\tau}(\y_t-\widetilde{\y}_{t+1})^\top (\y-\widetilde{\y}_{t+1}) \\
&=\frac{1}{\tau} (\y_t-\widetilde{\y}_{t+1} )^\top (\y_t - \widetilde{\y}_{t+1}) + \frac{1}{\tau} (\y_t-\widetilde{\y}_{t+1})^\top (\y - \y_t)  \\    
&= \frac{1}{\tau} \|\y_t - \widetilde{\y}_{t+1}\|^2 + \frac{1}{\tau} (\y_t-\widetilde{\y}_{t+1})^\top (\y-\y_t).
\end{split} 
\label{equ:local:proj_handle}
\end{equation} 

Combining (\ref{equ:local:proj_pause}) and (\ref{equ:local:proj_handle}),
\begin{align*}
g(\x_t, \y)\geq & g(\x_t,  \widetilde\y_{t+1})  + \frac{1}{\tau} \|\y_t - \widetilde\y_{t+1}\|^2 + \frac{1}{\tau}( \y_t -  \widetilde\y_{t+1})^{\top}(\y - \y_{t}) \\
&+ (\nabla_y g(\x_t, \y_t) -\w_{t+1})^{\top}(\y -  \widetilde\y_{t+1})
+ \frac{\lambda}{2}\|\y - \y_t\|^2   -  \frac{L_g}{2}\| \widetilde\y_{t+1}- \y_t\|^2. 
\end{align*}

Thus,
\begin{align*}
&g(\x_t ,  \widetilde\y_{t+1})\geq g(\x_t, \y^*(\x_t))\\
&\geq g(\x_t,  \widetilde\y_{t+1})  + \frac{1}{\tau} (\y_t -  \widetilde \y_{t+1})^{\top}(\y^*(\x_t) - \y_{t}) + (\nabla_y g(\x_t, \y_t) -\w_{t+1})^{\top}(\y^*(\x_t)-  \widetilde\y_{t+1}) \\
&~~~ + \frac{\lambda}{2}\|\y^*(\x_t) - \y_t\|^2   -  \frac{L_g}{2}\|\widetilde\y_{t+1}- \y_t\|^2 + \frac{1}{\tau}\|\y_t -  \widetilde\y_{t+1}\|^2\\
&\geq g(\x_t,  \widetilde\y_{t+1}) +\frac{ 1}{\tau} (\y_t -  \widetilde\y_{t+1})^{\top}( \y^*(\x_t) - \y_t)\\
&~~~ -  \frac{2}{\lambda}\|\nabla_y g(\x_t, \y_t) -\w_{t+1}\|^2 - \frac{\lambda}{4}\| \y^*(\x_t) - \y_t\|^2  - \frac{\lambda}{4}\|  \widetilde\y_{t+1}- \y_t\|^2\\
&~~~ + \frac{\lambda}{2}\| \y^*(\x_t) - \y_t\|^2 -  \frac{L_g}{2}\|\widetilde\y_{t+1} - \y_t\|^2+ \frac{1}{\tau}\|\y_t - \widetilde\y_{t+1}\|^2\\
&\geq g(\x_t,  \widetilde\y_{t+1}) +\frac{ 1}{2\tau_t\tau}(\|\y_{t+1}- \y^*(\x_t) \|^2 - \|\y_t - \y^*(\x_t)\|^2 -\tau_t^2 \|\y_t -  \widetilde\y_{t+1}\|^2 ) -  \frac{2}{\lambda}\|\nabla_y g(\x_t, \y_t) -\w_{t+1}\|^2\\
&~~~ + \frac{\lambda}{4}\| \y^*(\x_t) - \y_t\|^2 -  \frac{L_g}{2}\|\widetilde\y_{t+1} - \y_t\|^2+ \frac{1}{\tau}\|\y_t - \widetilde\y_{t+1}\|^2 - \frac{\lambda}{4}\|  \widetilde\y_{t+1}- \y_t\|^2.
\end{align*}
Hence we have
\begin{align*}
&\|\y_{t+1}- \y^*(\x_t) \|^2 \leq (1 - \frac{\tau_t\tau \lambda}{2})\|\y_t - \y^*(\x_t)\|^2 + \frac{4\tau_t\tau}{\lambda}\|\nabla_y g(\x_t, \y_t) -\w_{t+1}\|^2 \\ 
&  - 2\tau_t \tau(\frac{1}{\tau} - \frac{\tau_t}{\tau} - \frac{L_g}{2} - \frac{\lambda}{4 })\|\y_t-\widetilde\y_{t+1}\|^2\\
& \leq (1 - \frac{\tau_t\tau \lambda}{2})\|\y_t - \y^*(\x_t)\|^2 + \frac{4\tau_t\tau}{\lambda}\|\nabla_y g(\x_t, \y_t) -\w_{t+1}\|^2 \\
&  - 2\tau_t \tau(\frac{1}{2\tau} - \frac{3L_g}{4} )\|\y_t- \widetilde\y_{t+1}\|^2
\end{align*}
where we use $\tau_t\leq 1/2, \lambda\leq L_g$ and $\tau\leq 2/(3L_g)$. 
As a result, 
\begin{align*}
&\|\y_{t+1} - \y^*(\x_{t+1})\|^2 = \|\y_{t+1}- \y^*(\x_t) + \y^*(\x_t) - \y^*(\x_{t+1})\|^2\\
& \leq (1+ \frac{\tau_t\tau \lambda}{4})\|\y_{t+1}- \y^*(\x_t) \|^2 + (1+ \frac{4}{\tau_t\tau \lambda})\|\y^*(\x_t) - \y^*(\x_{t+1})\|^2\\
& \leq (1+ \frac{\tau_t\tau \lambda}{4})\|\y_{t+1}- \y^*(\x_t) \|^2 + (1+\frac{4}{\tau_t\tau \lambda})L_y^2\|\x_t - \x_{t+1}\|^2\\ 
& \leq  (1 - \frac{\tau_t\tau \lambda}{4})\|\y_t - \y^*(\x_t)\|^2 + \frac{8\tau_t\tau}{\lambda}\|\nabla_y g(\x_t, \y_t) -\w_{t+1}\|^2  +  \frac{8L_y^2}{\tau_t\tau\lambda}\|\x_t - \x_{t+1}\|^2\\
&~~~ -  (1+ \frac{\tau_t\tau \lambda}{4})2\tau_t \tau(\frac{1}{2\tau} - \frac{3L_g}{4} )\|\y_t- \widetilde\y_{t+1}\|^2\\
& \leq  (1 - \frac{\tau_t\tau \lambda}{4})\|\y_t - \y^*(\x_t)\|^2 + \frac{8\tau_t\tau}{\lambda}\|\nabla_y g(\x_t, \y_t) -\w_{t+1}\|^2  + \frac{8L_y^2}{\tau_t\tau\lambda}\|\x_t - \x_{t+1}\|^2\\
&~~~ -  (1+ \frac{\tau_t\tau \lambda}{4})\frac{2\tau}{\tau_t}(\frac{1}{2\tau} - \frac{3L_g}{4} )\|\y_t- \y_{t+1}\|^2
\end{align*}
where we use $(1+\epsilon/2)(1-\epsilon)\leq (1 - \epsilon/2 -  \epsilon^2/2)\leq 1 - \epsilon/2$.

\end{proof}

\subsection{Proof of Lemma~\ref{lem:2-2}}
\begin{align*}
&\|\d_{t+1}  - \nabla F(\x_t)\|^2\leq \|\d_{t+1} - \frac{1}{m}\sum_{i=1}^m\z_{i,t+1} + \frac{1}{m}\sum_{i=1}^m\z_{i,t+1}   - \nabla F(\x_t)\|^2 \\
&\leq 2\|\d_{t+1} - \frac{1}{m}\sum_{i=1}^m\z_{i,t+1}\|^2 + 2\| \frac{1}{m}\sum_{i=1}^m(\z_{i,t+1}   - \nabla F_i(\x_t))\|^2\\
&\leq 2\|\d_{t+1} - \frac{1}{m}\sum_{i=1}^m\z_{i,t+1}\|^2 + 2\frac{1}{m}\sum_{i=1}^m\| \z_{i,t+1}   - \nabla F_i(\x_t)\|^2\\ 
&\leq 2\|\d_{t+1} -  \frac{1}{m}\sum_{i=1}^m\z_{i,t+1}\|^2 
+\frac{1}{m}\sum_{i=1}^m4C_0\|\y_{i,t} - \y_i^*(\x_{t})\|^2\\ 
&~~~+  \frac{1}{m}\sum_{i=1}^m(4C_1\|\u_{i,t+1} -  \nabla_x f_i(\x_t, \y_{i,t}) \|^2+4C_2\|V_{i,t+1} - \nabla^2_{xy}g_i(\x_t, \y_{i,t})\|^2 ) \\
& ~~~ +\frac{1}{m}\sum_{i=1}^m (4C_3\|H_{i,t+1} - \nabla_{yy}^2g_i(\x_t, \y_{i,t})\|^2 +4C_4\|\v_{i,t+1} -\nabla_y f_i(\x_t, \y_{i,t}) \|^2), 
\end{align*} 
where the last inequality can be proven as in Lemma \ref{lem:2} with 
$C_0=\max\limits_{1\leq i \leq m} (2L_{fx,i}^2 +  \frac{6C_{fy,i}^2L_{gxy,i}^2}{\lambda^2} + \frac{6C_{fy,i}^2C^2_{gxy,i}L^2_{gyy,i}}{\lambda^4}+ \frac{6L_{fy,i}^2C^2_{gxy,i}}{\lambda^2})$, $C_1 = 2$, $C_2 =\max\limits_{1\leq i \leq m} ( \frac{6C_{fy,i}^2}{\lambda^2})$, $C_3 = \max\limits_{1\leq i \leq m} (\frac{6C_{fy,i}^2C_{gxy,i}^2}{\lambda^4})$ and $C_4 = \max\limits_{1\leq i \leq m} (\frac{6C_{gxy,i}^2}{\lambda^2})$. 

\subsection{Proof of Lemma~\ref{lem:11}}
\begin{proof}
\begin{align*}
&\E_{j_t}[\|\d_{t+1} - \frac{1}{m} \sum_{i=1}^m \z_{i,t+1})\|^2 ]\\
&=\E_{j_t}[\bigg\| (1-\beta) (\d_t -  \frac{1}{m} \sum_{i=1}^m \z_{i,t}) + \beta(\z_{j_t,t+1} - \frac{1}{m} \sum_{i=1}^m \z_{i,t+1}) )\\
& + (1-\beta)(\z_{j_t, t+1} - \z_{j_t, t} - \frac{1}{m} \sum_{i=1}^m \z_{i,t+1} + \frac{1}{m} \sum_{i=1}^m \z_{i,t}) \bigg\|^2]\\
& \leq (1-\beta)^2 \|\d_{t} - \frac{1}{m} \sum_{i=1}^m \z_{i,t}\|^2  + 2\beta^2\sigma_z^2 + 4 (1-\beta)^2\E_{i_t}[\|\z_{j_t, t+1} - \z_{j_t, t}\|^2]\\
& \leq (1-\beta)^2 \|\d_{t} - \frac{1}{m} \sum_{i=1}^m \z_{i,t}\|^2  + 2\beta^2\sigma_z^2\\
&~~~ + 4(1-\beta)^2\E_{j_t}[\|\u_{j_t, t+1} - V_{j_t, t+1}H^{-1}_{j_t, t+1}\v_{j_t, t+1} - \u_{j_t, t} + V_{j_t, t}H^{-1}_{j_t, t}\v_{j_t, t}\|^2]\\
& \leq (1-\beta)^2 \|\d_{t} - \frac{1}{m} \sum_{i=1}^m \z_{i,t}\|^2  + 2\beta^2\sigma_z^2\\
&~~~ + C'_1\E_{j_t}[\|\u_{j_t, t+1} - \u_{j_t, t}\|^2]+ C'_2\E_{j_t}[\|V_{j_t, t+1} - V_{j_t, t}\|^2] \\
&~~~ + C'_3\E_{j_t}[\|H_{j_t, t+1} - H_{j_t, t}\|^2] +  C_4'\E_{j_t}[\|\v_{j_t, t+1} - \v_{j_t, t}\|^2],
\end{align*}
where $C_1' = 8$, $C_2'=\max\limits_{1\leq i\leq m}\frac{24C^2_{fy,i}}{\lambda}$, $C_3' = \max\limits_{1\leq i\leq m} \frac{24 C^2_{gxy,i}C^2_{fy,i}}{\lambda^4} $, $C_4'=\max\limits_{1\leq i\leq m} \frac{24C^2_{gxy,i}}{\lambda^2}$.
\end{proof}

\subsection{Proof of Lemma~\ref{lem:12}}
\begin{proof}
\begin{align*}
&\E_{\xi_t, i_t}[\|\u_{t+1}-\u_t\|^2]   \leq  \E_{\xi_t, i_t}\|-\beta_t(\u_t - \widetilde h(\e_t; \xi_t; i_t)) + (1-\beta)(\widetilde h(\e_t; \xi_t; i_t) - \widetilde h(\e_{t-1}; \xi_t; i_t))\|^2\\
& \leq  2\beta_t^2 \E_{\xi_t, i_t} \|\u_t - \widetilde h(\e_t; \xi_t; i_t)\|^2 + 2(1-\beta_t)^2L(\p)\|\e_t - \e_{t-1}\|^2\\
& \leq  2\beta_t^2 \E_{\xi_t, i_t} \|\u_t  - h(\e_{t-1}) + h(\e_{t-1})- h(\e_t) + h(\e_t)- \widetilde h(\e_t; \xi_t; i_t)\|^2 + 2(1-\beta_t)^2L(\p)\|\e_t - \e_{t-1}\|^2\\
& \leq  4\beta_t^2\|\u_t  - h(\e_{t-1})\|^2 + 4\beta_t^2\|h(\e_{t-1})- h(\e_t)\|^2 + 2\beta_t^2C(\p)+ 2(1-\beta_t)^2L(\p)\|\e_t - \e_{t-1}\|^2\\
& \leq  4\beta_t^2\delta_{t-1}^2+ 4\beta_t^2\sum_{i}L_i\|\e_{t-1} - \e_t\|^2 + 2\beta_t^2G(\p)+  2(1-\beta_t)^2L(\p)\|\e_t - \e_{t-1}\|^2\\
& \leq  4\beta_t^2\delta_{t-1}^2 + 2\beta_t^2G(\p)+ 4(L(\p)+\beta_t^2 \sum_iL_i)\|\e_t - \e_{t-1}\|^2. 
\end{align*}
\end{proof}

\section{Preliminary Experiments}
In this section, we present some preliminary experiments to verify the effectiveness of the proposed algorithms. 
\subsection{One Lower-level Problem}
We consider an application of hyper-parameter optimization in machine learning (ML). A potential problem in ML is that the distribution of training examples might not be the same as the testing data. In this case, it could be beneficial to assign non-uniform weights to training examples to account for the difference between training set and testing set. The non-uniform weights will be learned to minimize the error on a separate validation data. 
In particular, we consider the following formulation: 
\begin{equation*}
\min_{\p}F(\p):=f(\p,\w^*(\p)):=\frac{1}{\left|\D_{va} \right|}\sum_{\xi_j\in \D_{va}} \ell(\w^*(\p);\xi_j)
\end{equation*}
\begin{equation*}
\w^*(\p)\in \arg \min_{\w} g(\p,\w):= \frac{1}{\left|\D_{tr} \right|} \sum_{\zeta_j \in \D_{tr}} \p_j \ell(\w;\zeta_j)+\frac{\lambda}{2}\|\w\|^2
\end{equation*}
where $\D_{tr}$ and $\D_{va}$ are the training and validation dataset, and $\ell(\w;\xi)$ is the logistic loss for clasification. To evaluate the algorithms, we consider two UCI Adult benchmark datasets (a1a, a8a) and two web page classification datasets (w1a, w8a) ~\citep{Dua2019,Platt99}. The prediction task of the Adult datasets is to determine whether a person makes over 50K a year based on $123$ binary attributes. Here a1a and a8a stand for different training and testing separation of the same data set. The webpage classification task is a text categorization problem, of which each input contains $300$ sparse binary keyword attributes extracted from each web page. Note that w1a and w8a differ in the similar way to a1a and a8a.

For our experiment, we use the testing test of each datas as the validation data $\D_{va}$ for evaluating the convergence speed of different algorithms. For all algorithms, we use two step sizes, one  step size $\eta_t$ updating the main variable $\x$ and another step size $\tau_t$ for updating the variable $\y$ for the lower problem. For all baselines except for STABLE, we set $\eta_t = (t+c_0)^{-1/2}$ and $\tau_t = \tau_0\eta_t^2$,  for STABLE, we use $\tau_t = \tau_0\eta_t$. The reason is that STABLE is the single time-scale algorithm.  For SVRB, we use step sizes $\eta_t =\tau_t = (t+c_0)^{-1/3}$ and use the same parameter $\beta_t = \beta\eta_t^2$ for updating all gradient estimators. The constants $c_0$, $\tau_0$ and $\beta$ are tuned in a range. For SVRB and STABLE, the Hessian inverse times a vector is implemented by the conjugated gradient algorithm. 

We report the results in Figure~\ref{fig:onelevel}, where the top panel shows the validation loss (objective) vs the number of samples  and the bottom panel shows validation loss (objective) vs running time. We can see tht the proposed SVRB achieves the fastest convergence rate among all algorithms in terms of running time and sample complexity.

\begin{figure}
    \centering
    \begin{subfigure}[b]{0.24\textwidth}
        \includegraphics[width=\textwidth]{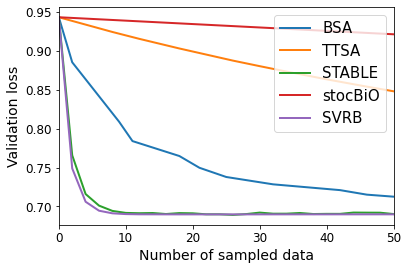}
        \includegraphics[width=\textwidth]{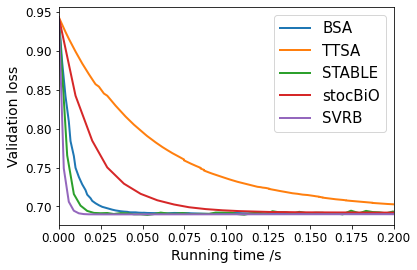}
        \caption{Dataset: a1a}
    \end{subfigure}
    \begin{subfigure}[b]{0.24\textwidth}
        \includegraphics[width=\textwidth]{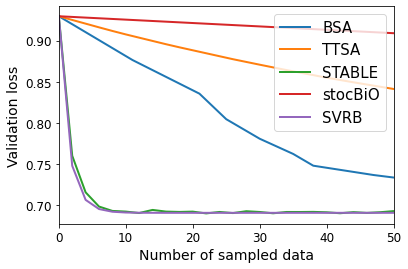}
        \includegraphics[width=\textwidth]{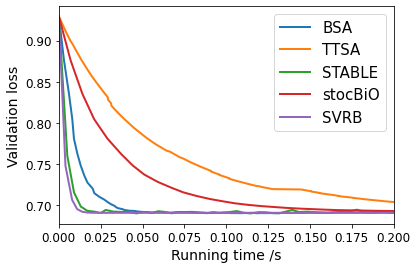}
        \caption{Dataset: a8a}
    \end{subfigure}
    \begin{subfigure}[b]{0.24\textwidth}
        \includegraphics[width=\textwidth]{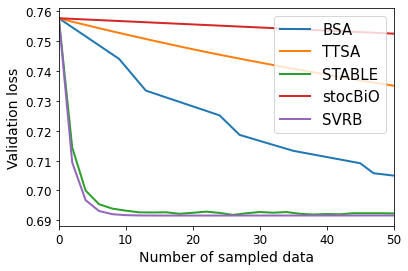}
        \includegraphics[width=\textwidth]{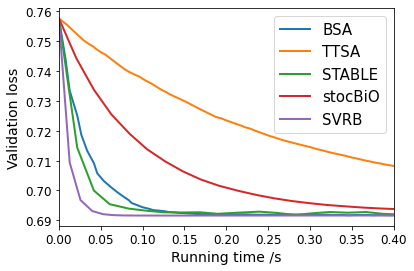}
        \caption{Dataset: w1a}
    \end{subfigure}
    \begin{subfigure}[b]{0.24\textwidth}
        \includegraphics[width=\textwidth]{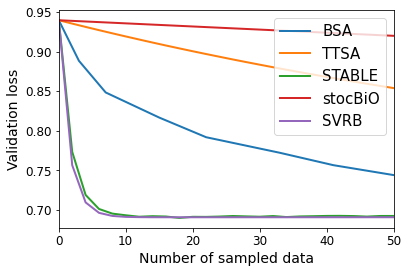}
        \includegraphics[width=\textwidth]{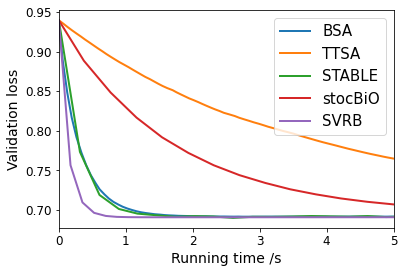}
        \caption{Dataset: w8a}
    \end{subfigure}
     \caption{Comparison of various algorithms on logistic loss problem over UCI Adult datasets and webpage classification task. For each data set, an objective function value v.s. running time plot and an objective function value v.s. sample complexity plot are given.}
     \label{fig:onelevel}
\end{figure}

\subsection{Many Lower-level Problems}
To conduct an experiment on the proposed algorithm RE-RSVRB for SBO with many lower-level problems, we consider an extension of hyper-parameter optimization with many different loss functions: 
\begin{equation*}
\min_{\p}F(\p):=\sum_{i=1}^m f(\p,\w_i^*(\p)):=\frac{1}{\left|\D_{va} \right|}\sum_{i=1}^m\sum_{\xi_j\in \D_{va}} \mathcal{L}_i(\w_i^*(\p);\xi_j)
\end{equation*}
\begin{equation*}
\w_i^*(\p)\in \arg \min_\w g_i(\p,\w):= \frac{1}{\left|\D_{tr} \right|} \sum_{\zeta_j \in \D_{tr}} \p_j \mathcal{L}_i(\w;\zeta_j)+\frac{\lambda}{2}\|\w\|^2,\quad i=1,\dots,m
\end{equation*}
where the loss function is defined as
\begin{equation*}
    \mathcal{L}_i(\w;x,y):=\log\left(1+\exp\left(\frac{-y (\w^Tx+w_0)}{\sigma_i}\right)\right).
\end{equation*}
where $\sigma_i$ is a temperature parameter. 
The step sizes are chosen in the similar way to the experiment for SBO with one lower-level problem. 

For RE-RSVRB, we sample multiple lower problems (10 tasks) at each iteration for updating the model. 
We test the baselines and RE-RSVRB on two datasets, namely a1a and w1a. To control the time spent on the objective function value evaluations, we only choose $100$ data points from the testing part of each dataset as the validation data. Note that the size of validation data has minimal impact on the performance of the algorithms. The starting step size is tuned between $1/1000$ and $1/10$, and the constant $\tau$ in the update of $\y_i$ is chosen from $\{0.1, 0.2,0.3,0.4,0.5,0.6,0.7,0.8,0.9,1\}$. The constants $\sigma_i$ in the loss function $\mathcal{L}_i$ are randomly generated from the interval $[1,11]$ and fixed. We use D-m as a shorthand of data set D with $m$ lower-level problems. It can be seen from Figure~\ref{fig:manylevel} that RE-RSVRB outperforms the baselines in most cases.
\begin{figure}
    \centering
    \begin{subfigure}[b]{0.24\textwidth}
        \includegraphics[width=\textwidth]{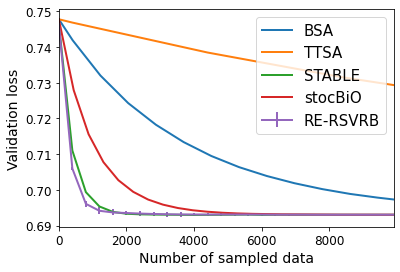}
        \includegraphics[width=\textwidth]{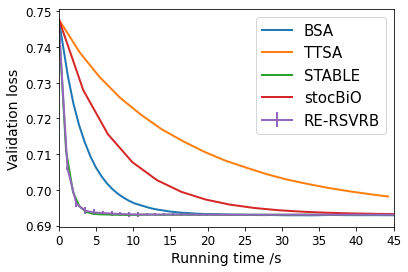}
        \caption{a1a-200}
    \end{subfigure} 
    \begin{subfigure}[b]{0.24\textwidth}
        \includegraphics[width=\textwidth]{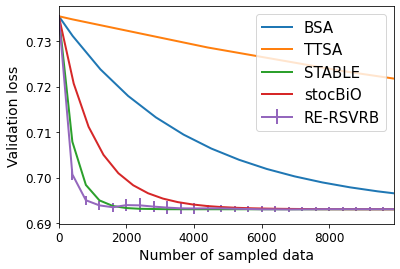}
        \includegraphics[width=\textwidth]{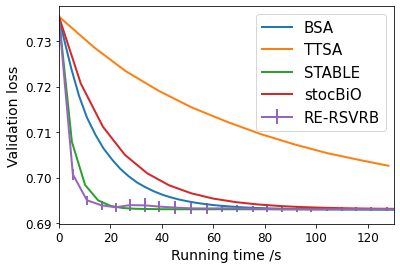}
        \caption{w1a-200}
    \end{subfigure}
    \begin{subfigure}[b]{0.24\textwidth}
        \includegraphics[width=\textwidth]{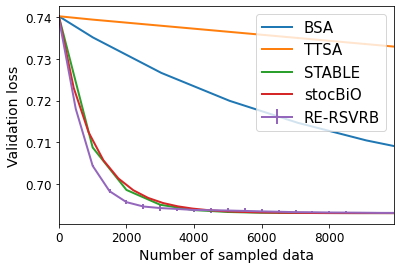}
        \includegraphics[width=\textwidth]{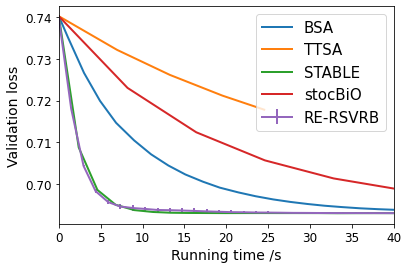}
        \caption{a1a-500}
    \end{subfigure}
    \begin{subfigure}[b]{0.24\textwidth}
        \includegraphics[width=\textwidth]{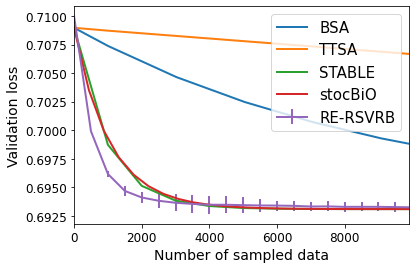}
        \includegraphics[width=\textwidth]{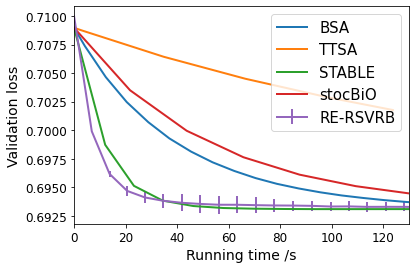}
        \caption{w1a-500}
    \end{subfigure}
    \caption{Comparison of various algorithms on SBO with $200$ or $500$ lower-level problems over UCI Adult data set and web classification task. For each setting, an objective function value v.s. running time plot and an objective function value v.s. sample complexity plot are given.}
    \label{fig:manylevel}
\end{figure}
\end{document}